\documentclass[12pt]{article}
\usepackage[paperwidth=210mm, paperheight=297mm,bindingoffset=0mm]{geometry}
\usepackage[T1]{fontenc}
\usepackage{setspace}
\usepackage{tgtermes}
\usepackage{bbm}
\usepackage{mathbbol}
\usepackage[]{enumitem}
\setlist[enumerate,1]{itemsep=0pt, parsep=0pt, listparindent=\parindent}
\setlist[enumerate,2]{ref=\theenumii, itemsep=0pt, parsep=0pt, listparindent=\parindent}
\setlist[itemize,1]{itemsep=0pt, parsep=0pt, listparindent=\parindent}
\setlist[itemize,2]{itemsep=0pt, parsep=0pt, listparindent=\parindent}
\usepackage{titling}
\usepackage{hyperref}
\usepackage{mathtools}
\usepackage[backend=biber,
style=numeric,
sorting=ynt
]{biblatex}
\usepackage[nottoc,numbib]{tocbibind}
\usepackage{listings}

\usepackage{xcolor}
\usepackage[bottom]{footmisc} %to place footnote at bottom of the page 
\usepackage{amssymb}
\usepackage{amsthm}
\usepackage[protrusion=true,expansion=true]{microtype}
\usepackage{blindtext}
\usepackage{etoolbox}
\usepackage{graphicx}              % to include figures
\usepackage{amsmath}               % great math stuff
\usepackage{amsfonts}              % for blackboard bold, etc
\usepackage{amsthm}                % better theorem environments
\usepackage{hyperref}% http://ctan.org/pkg/hyperref
\hypersetup{citecolor=red}
\usepackage{cleveref}
\numberwithin{equation}{section}
\usepackage{caption}
\usepackage{float}

\hypersetup{
colorlinks=true,
linkcolor=blue
}

\title{Forcing as a Local Method of Accessing Small Extensions}
\author{Desmond Lau}

\begin{document}

%\onehalfspacing
\maketitle

\begin{abstract}
Fix a set-theoretic universe $V$. We look at small extensions of $V$ as generalised degrees of computability over $V$. We also formalise and investigate the complexity of certain methods one can use to define, in $V$, subclasses of degrees over $V$. Finally, we give a nice characterisation of the complexity of forcing within this framework.
\end{abstract}

\newtheorem{thm}{Theorem}[section]
\newtheorem{innercustomlem}{Lemma}
\newenvironment{customlem}[1]
  {\renewcommand\theinnercustomlem{#1}\innercustomlem}
  {\endinnercustomlem}
\newtheorem{innercustomdef}{Definition}
\newenvironment{customdef}[1]
  {\renewcommand\theinnercustomdef{#1}\innercustomdef}
  {\endinnercustomdef}
\newtheorem{lem}[thm]{Lemma}
\newtheorem{prop}[thm]{Proposition}
\newtheorem{cor}[thm]{Corollary}
\newtheorem{conj}[thm]{Conjecture}
\newtheorem{ques}[thm]{Question}
\newtheorem*{claim}{Claim}
\theoremstyle{definition}
\newtheorem{defi}[thm]{Definition}
\theoremstyle{remark}
\newtheorem*{rem*}{Remark}
\newtheorem{rem}[thm]{Remark}
\newtheorem{ex}[thm]{Example}
\newtheorem{ob}[thm]{Observation}
\newtheorem{fact}[thm]{Fact}
\newtheorem{con}[thm]{Convention}
\newtheorem{diff}[thm]{Difficulty}

\newcommand{\bd}[1]{\mathbf{#1}}  % for bolding symbols
\newcommand{\RR}{\mathbb{R}}      % for Real numbers
\newcommand{\ZZ}{\mathbb{Z}}      % for Integers
\newcommand{\col}[1]{\left[\begin{matrix} #1 \end{matrix} \right]}
\newcommand{\comb}[2]{\binom{#1^2 + #2^2}{#1+#2}}
\newcommand{\eq}{=}

\newcommand{\blankpage}{
\newpage
\thispagestyle{empty}
\mbox{}
\newpage
}

{\let\clearpage\relax \tableofcontents} 
\thispagestyle{empty}

\section{Introduction}\label{sect1}

According to \cite{rasm}, there is an intuitive notion of generalised computation, definable within the set-theoretic universe $V$, that canonically partitions sets into their degrees of constructibility. This lends credence to the belief that generative power over a model of set theory is a surrogate for computational power. When dealing with degrees of constructibility, the relevant model of set theory is G\"{o}del's constructible universe $L$, an inner model of $V$. Switching out $L$ for larger inner models makes sense for coarser degree structures.

What if we swop $L$ for $V$ itself? Doing so will obviously result in degrees that are not subclasses of $V$. What then do they comprise? With meta-theoretic assumptions mildly stronger than $\mathsf{ZFC}$, we can view $V$ as a countable transitive model of $\mathsf{ZFC}$, from which such degrees can be naturally defined as degrees of small extensions. Vaguely, each degree of small extensions is associated with (or rather, represented by) an outer model $W$ of $V$ generated by a set in $W$ over $V$: here $W$ is called a small extension of $V$. These degrees, together with the theory of their ordering, seek to capture the spirit of higher-order computations relative to $V$, the way higher recursion theory do for computations on sets beyond the domain of classical recursion theory. Figure \ref{analogy} illustrates this parallel.

\begin{figure}[!ht]
    \centering
    \includegraphics[width=\textwidth]{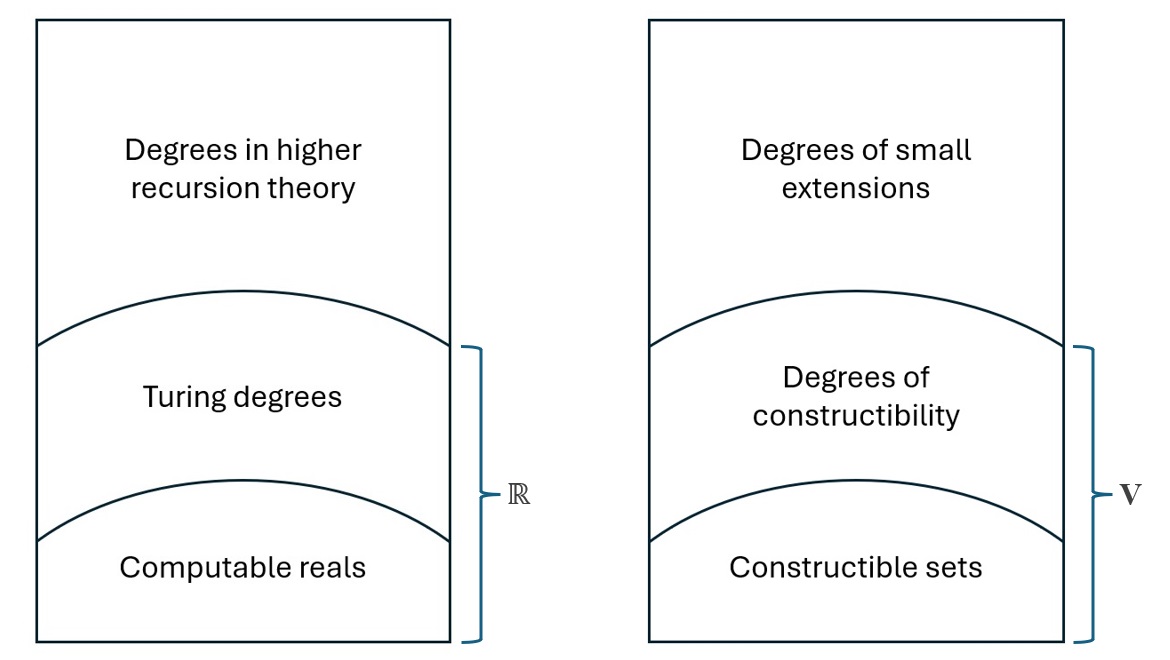}
    \captionsetup{width=0.8\textwidth, format=hang}
    \caption{comparison between conventional notions of relative computability (left) and our generalised notions (right).}
    \label{analogy}
\end{figure}

Next, we wish to examine (necessarily non-constructive) methods of definably ``accessing'' small extensions of $V$ within $V$, or local methods in short. Set forcing is one such method, and a very well-studied one at that. In an application of set forcing, we pick a partially ordered set --- also known as a forcing notion --- $\mathbb{P} \in V$, and use a filter meeting all dense subsets of $\mathbb{P}$ in $V$ --- termed a $\mathbb{P}$-generic filter over $V$ --- to generate an extension of $V$. So the small extensions of $V$ set forcing brings about via $\mathbb{P}$ are precisely those in
\begin{equation*}
    \{V[g] : g \text{ is a } \mathbb{P} \text{-generic filter over } V\} \text{,}
\end{equation*}
a set definable outside $V$ with $\mathbb{P}$ and its dense subsets in $V$ as parameters. Consequently, one can view set forcing as a recipe in $V$ for generating small extensions of $V$ based only on parameters in $V$. The formal treatment of set forcing inspires a list of desiderata for a local method:
\begin{enumerate}[leftmargin=40pt, label=(DA\arabic*)]
    \item\label{da1} it should be definable in $V$,
    \item\label{da2} it should map parameters to descriptions of how those parameters are used to define generators of small extensions, and
    \item\label{da3} the generators it produces should depend locally on the parameters used to define them. 
\end{enumerate}

A convenient realisation at this juncture is that a recipe and its parameters (or equivalently, the two components of \ref{da2}) can be bundled up into a theory with constraints in interpretation (TCI). TCIs are basically first-order theories endowed with set constraints that may not be first-order expressible. Like standard first-order theories, TCIs admit models, and whether a set $X$ is a model of a TCI $\mathfrak{T}$ depends locally on $X$ and $\mathfrak{T}$. Defining a local method through the language of TCIs and their models thus provides immediate guarantee of \ref{da3}, and is appealing in both its brevity and robustness.

Accompanying the formalisation of local methods, ought to be a notion of relative complexity, a measure which can be utilised to check if one local method is ``more complex'' than another. Akin to relative computability, we want to define relative complexity as a transitive binary relation on the class of all local methods. There is actually a straightforward way to do this: we say method $Y$ is more complex than method $X$ iff the small extensions of $V$ picked out by $Y$ are a non-trivial refinement of those picked out by $X$. Connecting the first-order portion of a TCI with the relative complexity relation we defined, leads to the formulation of a complexity hierarchy --- the local method hierarchy --- very much in line with more notable hierarchies in theoretical computer science (e.g. the arithmetical and polynomial hierarchies).

Leveraging on a novel forcing framework developed in \cite{myself}, we are able to show that the method of set forcing is exactly $\mathsf{\Sigma_1}$ (or equivalently, as we shall see, $\mathsf{\Pi_2}$) in the local method hierarchy. This is the main takeaway of our work presented here. We follow it up with the analysis of certain witnesses to set forcing being more complex than $\mathsf{\Pi_2}$.

By applying an analogue of the Cantor-Bendixson derivative on a specific class of forcing notions, we prove that every TCI $\mathfrak{T} \in \mathsf{\Pi_2}$ either singles out $V$ or picks out continuum-many (as evaluated in the meta-theory) small extensions of $V$. The same is long known to be true for forcing notions: a trivial forcing notion gives $V$ as its sole generic extension, whereas a non-trivial one generates continuum-many generic extensions.

One can think of this work as a rigorous foundation for some of the main ideas found in Section 5 of \cite{myself}. In fact, consolidated here are many results in said section, having been weaved together into a more philosophically compelling and coherent package. For self-containment, we reproduce the more concise proofs from \cite{myself}, of as many of these results as possible.  

\section{Degrees of Small Extensions}\label{sect3}

Extending a structure through means of adjoining ``new'' objects is commonly seen in mathematics. Here, ``new'' just means ``existing outside of the structure in question''. For example, a standard course in algebra would talk about field extensions the likes of $\mathbb{Q}[\sqrt{2}]$. In set theory, the subjects of study are models of set theory, often models of $\mathsf{ZFC}$. For convenience, we usually assume such models are countable and transitive. If a \emph{countable transitive model of} $\mathsf{ZFC}$ (henceforth, \emph{CTM}) exists, then extensions of it exist, but due to the complicated closure properties required of a model of $\mathsf{ZFC}$, the proof of their existence is much hairier than that of field extensions. 

It turns out that, whenever $U$ is a CTM and $W$ is an extension of $U$, we can always find an extension of $W$ that is generated over $U$ by a ``small set''. Methods of generating such ``small extensions'' include, but are not limited to, set forcing. In this section, we compare the multiverse of small extensions with the \emph{generic multiverse} born from set forcing, under the assumption that both multiverses have the same centre. We also introduce the idea of \emph{theories with constraints in interpretation} (\emph{TCIs}) to set things up for the next section. 

\subsection{Small Extensions as Degrees}

\begin{rem}
We avoid the usual meta-theoretic concerns regarding forcing and the set-theoretic multiverse by working in the theory 
\begin{equation*}
    \mathsf{ZFC} \ + \text{ ``there is a transitive model of } \mathsf{ZFC}\text{''.}
\end{equation*} 
The existence of CTMs can be proven in this theory.
\end{rem}

\begin{defi}
Given $U$ and $W$, we say $U$ is an \emph{inner model} of $W$ (or equivalently, $W$ is a \emph{outer model} of $U$) iff
\begin{itemize}
    \item $U$ and $W$ are CTMs,
    \item $U \subset W$, and
    \item $ORD^U = ORD^W$.
\end{itemize}
\end{defi}

\begin{defi}
Let $W$ be an outer model of $U$. Then $W$ is a \emph{small extension of} $U$ iff for some $x \in W$, $W$ is the smallest CTM $W'$ satisfying
\begin{itemize}
    \item $U \subset W'$, and
    \item $x \in W'$.
\end{itemize}
In this case, we say, equivalently,
\begin{itemize}
    \item $x$ \emph{generates} $W$ \emph{from} $U$, or
    \item $W$ \emph{is a small extension of} $U$ \emph{generated by} $x$, or
    \item $W = U[x]$.
\end{itemize}
\end{defi}

The following observation is trivial.

\begin{ob}
The binary relations 
\begin{itemize}
    \item ``being an outer model of'', and
    \item ``being a small extension of''
\end{itemize}
are both transitive.
\end{ob}

We know that certain sets in outer models can always be used to generate small extensions.

\begin{fact}\label{fact34}
Let $W$ be an outer model of $U$, and $x \in \mathcal{P}(y) \cap W$ for some $y \in U$. Then there is a smallest CTM $W'$ satisfying
\begin{itemize}
    \item $U \subset W'$, and
    \item $x \in W'$.
\end{itemize}
In other words, $U[x]$ exists.
\end{fact}

There is a simple and useful characterisation of small extensions of a CTM.

\begin{prop}\label{prop22}
Let $M$ be a transitive model of $\mathsf{ZFC}$ and $X \in M$. Then there is a set of ordinals $c \in M$ such that if $N$ is any transitive model of $\mathsf{ZFC}$ containing $c$, then $X \in N$. 
\end{prop}

\begin{proof}
Let $Y'$ be the transitive closure of $X$ (under the membership relation $\in$) and set $Y := Y' \cup \{X\}$. Then $Y$ is $\in$-transitive. Choose a bijection $f$ from a cardinal $\kappa$ into $Y$. Use $\in'$ to denote the unique binary relation $R$ on $\kappa$ such that
\begin{equation*}
    R(\alpha, \beta) \iff f(\alpha) \in f(\beta) \text{.}
\end{equation*}
Now apply G\"{o}del's pairing function to code $\in'$ as a (necessarily unbounded) subset $c$ of $\kappa$. 

To recover $X$ from $c$, first apply the inverse of the pairing function followed by the Mostowski collapse to get $Y$. Then $X$ is definable from $Y$ as the unique $\in$-maximal element of $Y$. This decoding process is absolute for transitive models of $\mathsf{ZFC}$ because all its components are.
\end{proof}

\begin{prop}\label{prop35}
If $W$ is a small extension of $U$, then for some ordinal $\kappa \in U$, $W$ is generated from $U$ by an unbounded subset of $\kappa$. Furthermore, we can choose $\kappa$ such that
\begin{equation*}
    (U; \in) \models \text{``} \kappa \text{ is a cardinal''.}
\end{equation*}
\end{prop}

\begin{proof}
Let $x \in W$ be such that $W = U[x]$. By the proof of Proposition \ref{prop22}, for some cardinal $\kappa \in W$ there is an unbounded subset $c$ of $\kappa$ in $W$ coding $x$. Since $U$ is an inner model of $W$, $\kappa$ is also a cardinal in $U$. By Fact \ref{fact34}, $U[c]$ exists. Now, $U[c] \subset W$ as $c \in W$; but also $W \subset U[c]$ because $c$ can be decoded in $U[c]$ to give $x$ and $W = U[x]$.
\end{proof}

\begin{defi}
Let $U$ be a CTM. The \emph{outward multiverse centred at} $U$ is the set
\begin{equation*}
    \mathbf{M}(U) := \{W : W \text{ is an outer model of } U\} \text{.}
\end{equation*}
The \emph{small outward multiverse centred at} $U$ is the set
\begin{equation*}
    \mathbf{M}_{S}(U) := \{W : W \text{ is a small extension of } U\} \text{.}
\end{equation*}
\end{defi}

Clearly, $\mathbf{M}_{S}(U) \subset \mathbf{M}(U)$. By Jensen's remarkable result on ``coding the universe'' into a real, $\mathbf{M}_{S}(U)$ is not that much smaller than $\mathbf{M}(U)$.

\begin{fact}[Jensen, \cite{jensencoding}]\label{fact37}
Every CTM has an outer model satisfying 
\begin{equation*}
    \text{``} V = L[r] \text{ for some } r \subset \omega \text{''.}
\end{equation*}
\end{fact}

\begin{prop}
Given a CTM $U$, $(\mathbf{M}_{S}(U), \subset)$ is a cofinal subposet of $(\mathbf{M}(U), \subset)$.
\end{prop}

\begin{proof}
Let $W \in \mathbf{M}(U)$. Then by Fact \ref{fact37}, there is $W' \in \mathbf{M}(U)$ such that $W \subset W'$ and $W'$ satisfies
\begin{equation*}
    \text{``} V = L[r] \text{ for some } r \subset \omega \text{''.}
\end{equation*}
Since $L^{W'} = L^U \subset U$ and indeed $W' = L^{W'}[r]$ for some real $r \in W'$ from the outside, necessarily $W' = U[r]$. But this means $W' \in \mathbf{M}_{S}(U)$.
\end{proof}

We can characterise members $\mathbf{M}_{S}(U)$ in a way that is conducive to the discussion of relative computability.

\begin{prop}\label{prop37}
Let $U$ be a CTM. Then
\begin{align*}
    \mathbf{M}_{S}(U) & = \{U[x] : x \in \bigcup \mathbf{M}(U) \cap \mathcal{P}(U)\} \\
    & = \{U[x] : x \in \bigcup \mathbf{M}(U) \cap \mathcal{P}(ORD^U)\} \text{.}
\end{align*}
\end{prop}

\begin{proof}
By Fact \ref{fact34} and Proposition \ref{prop35}.
\end{proof}

Proposition \ref{prop37} gives us a natural reducibility relation on 
\begin{equation*}
    \mathbf{N}(U) := \bigcup \mathbf{M}(U) \cap \mathcal{P}(ORD^U)
\end{equation*}
given a CTM $U$. 

\begin{defi}
Let $U$ be a CTM. Define the binary relation $\leq_U$ on $\mathbf{N}(U)$ as follows: for $x, y \in \mathbf{N}(U)$,
\begin{equation*}
    x \leq_U y \iff U[x] \subset U[y] \text{.}
\end{equation*}
Given $x, y \in \mathbf{N}(U)$, write $x \equiv_U y$ iff $x \leq_U y$ and $y \leq_U x$.
\end{defi}

One can easily check that $\leq_U$ is a preorder, so taking its quotient by $\equiv_U$ results in a partial order we shall denote as $(\mathcal{D}(U), \leq_{\mathcal{D}(U)})$.

The case of $(\mathcal{D}(U), \leq_{\mathcal{D}(U)})$ parallels that of the constructibility degrees, in the sense that the former partial order, like the latter one, is isomorphic to a class of set-theoretic universes under inclusion. Specifically, 
\begin{equation*}
   (\mathcal{D}(U), \leq_{\mathcal{D}(U)}) \cong (\mathbf{M}_S(U), \subset) \text{.} 
\end{equation*} 
This motivates viewing $(\mathcal{D}(U), \leq_{\mathcal{D}(U)})$ as degrees of computability over $U$. These degrees are necessarily non-constructible (and indeed, non-constructive) if $U$ is not a model of ``$V = L$''. Whereas the ``constructible in'' relation partially orders a partition of an $\in$-model of $\mathsf{ZFC}$ and is definable within said model, the field of $(\mathcal{D}(U), \leq_{\mathcal{D}(U)})$ may not be realisable as a partition of any such model. We thus expect the structure of $(\mathcal{D}(U), \leq_{\mathcal{D}(U)})$ to be much more varied and dependent on $U$, compared to the structure of the constructibility degrees evaluated in $U$. Nevertheless, we will attempt to stratify $(\mathcal{D}(U), \leq_{\mathcal{D}(U)})$.

Hereon, we shall analyse and reason about $(\mathcal{D}(U), \leq_{\mathcal{D}(U)})$ by moving to $(\mathbf{M}_S(U), \subset)$, so that we can apply set-theoretic arguments and leverage on set-theoretic techniques.

\subsection{Forcing and the Generic Multiverse}

Forcing is a technique invented by Cohen in \cite{cohen1} to prove that the continuum hypothesis is independent of $\mathsf{ZFC}$. It has since taken on a life of its own, becoming an indispensable tool in set theory, and even in other branches of logic. The modern treatment of forcing is largely due to Scott, Solovay, Silver, and Rowbottom, as communicated by Shoenfield in \cite{shoenfield}. We shall give a very brief and high-level introduction to forcing, following the layout found in Section 2.4 of \cite{myself}.

In a typical application of forcing, we start with a CTM, called the \emph{ground model}. The usual forcing argument can be rewritten to occur entirely in the ground model with respect to a \emph{forcing notion} that lives therein. Exactly because of this, we often forget the fact that our ground model is a CTM, or at least we eschew mentioning it. This is also why our ground model is conventionally taken to be $V$ itself. 

Forcing parlance dictates a forcing notion to just be a partial order. The crux of forcing is the analysis of generic filters (which may not exist in $V$) of a forcing notion $\mathbb{P} \in V$ via the \emph{forcing relation} $\Vdash_{\mathbb{P}}$ defined on $\mathbb{P}$ in $V$. Forcing relations are the trick to reasoning about extensions of $V$ without needing to step out of $V$.

\begin{defi}
Let $\mathbb{P} = (P, \leq_{\mathbb{P}})$ be a forcing notion, $D \subset P$ and $A$ be any set. We say a subset $g$ of $P$ \emph{meets} $D$ \emph{with respect to} $\mathbb{P}$ \emph{in} $A$ iff $$g \cap \{p \in P : p \in D \text{ or } \forall q \ (q \leq_{\mathbb{P}} p \implies q \not\in D)\} \cap A \neq \emptyset.$$ We say $g$ \emph{meets} $D$ \emph{with respect to} $\mathbb{P}$ iff $g$ meets $D$ with respect to $\mathbb{P}$ in $V$.
\end{defi}

\begin{defi}
Let $\mathbb{P} = (P, \leq_{\mathbb{P}})$ be a forcing notion and $\mathfrak{A} = (A; \in, \Vec{X})$ be a structure in a possibly expanded language of set theory. We say a subset $g$ of $P$ is $\mathbb{P}$\emph{-generic over} $\mathfrak{A}$ (or $g$ is a $\mathbb{P}$\emph{-generic subset over} $\mathfrak{A}$) iff $g$ meets $D$ with respect to $\mathbb{P}$ in $A$ for all $D$ such that
\begin{itemize}
    \item $D \subset P$
    \item $D$ is dense in $\mathbb{P}$, and
    \item $D$ is definable over $\mathfrak{A}$ with parameters in $A$.
\end{itemize}
If in addition, $g$ is a filter on $\mathbb{P}$, then we call $g$ a $\mathbb{P}$\emph{-generic filter over} $\mathfrak{A}$.
\end{defi}

\begin{defi}
Let $\mathbb{P}$ be a forcing notion and $\mathfrak{A}$ be a structure. A $(\mathbb{P}, \mathfrak{A})$\emph{-generic object} is a set definable from a $\mathbb{P}$-generic filter over $\mathfrak{A}$, with parameters from $\mathfrak{A}$.
\end{defi}

\begin{ob}\label{ob0}
Let $\mathbb{P} = (P, \leq_{\mathbb{P}})$ be a forcing notion and $X$ be any set. Then there is a structure $\mathfrak{A} = (A; \in) \in V$ such that in every outer model of $V$, 
\begin{align*}
    x \text{ is a } (\mathbb{P}, \mathfrak{A}) \text{-generic object} \iff x \text{ is a } (\mathbb{P}, V) \text{-generic object}
\end{align*}
for all $x \subset X$. In fact, we can choose $A$ to be $H(\kappa)$ for any $\kappa > |trcl(\{P, X\})|$.
\end{ob}

Given a forcing notion $\mathbb{P} = (P, \leq_{\mathbb{P}})$ in $V$, the forcing relation $\Vdash_{\mathbb{P}}^V$ is a binary relation that relates members of $P$ with formulas parametrised by members of $V^{\mathbb{P}}$, where $V^{\mathbb{P}}$ is the class of $\mathbb{P}$\emph{-names} in $V$. Both $V^{\mathbb{P}}$ and $\Vdash_{\mathbb{P}}^V$ are uniformly definable in $V$ over the class of all forcing notions $\mathbb{P}$. $\mathbb{P}$-names in $V$ are ``evaluated at" a $\mathbb{P}$-generic filter $g$ over $V$ to obtain the $\mathbb{P}$-generic extension $V[g]$, which is necessarily a CTM, and thus is a small extension of $V$. In more formal writing, if $g$ is a $\mathbb{P}$-generic filter over $V$, then
\begin{equation*}
    V \subset V[g] := \{\dot{x}[g] : \dot{x} \in V^{\mathbb{P}}\},
\end{equation*}
where $\dot{x}[g]$ means ``$x$ evaluated at $g$". Of course, this evaluation procedure is done outside $V$ because any such non-trivial $g$ would not exist in $V$. Even so, the ingenuity of forcing as a technique lies in the amount of knowledge we can deduce about $V[g]$ in $V$ through examining $\Vdash_{\mathbb{P}}^V$ alone.

\begin{con}
When it is clear that the background universe is $V$, we suppress mention of $V$ when writing forcing relations in $V$. This means that given a forcing notion $\mathbb{P}$ in $V$, $\Vdash_{\mathbb{P}}$ is used interchangeably with $\Vdash_{\mathbb{P}}^V$.
\end{con}

\begin{defi}
We call $W$ a \emph{forcing extension} (or a \emph{generic extension}) of $V$ iff there exists a forcing notion $\mathbb{P}$ in $V$ and a $\mathbb{P}$-generic filter $g$ over $V$, such that $W = V[g]$.
\end{defi}

\begin{defi}
We write ``$\Vdash_{\mathbb{P}} \phi$" to mean $$\text{``} \forall p \ (p \in \mathbb{P} \implies p \Vdash_{\mathbb{P}} \phi) \text{''.}$$ 
\end{defi}

The next theorem is important enough to be stated here in full, but not relevant enough to the spirit of this section to warrant a reproduction of its proof.

\begin{thm}\label{thm310}
If $\mathbb{P}$ is a forcing notion in $V$, $p \in \mathbb{P}$, $\phi$ is a formula with $n$ free variables, and $\dot{x}_1$, ..., $\dot{x}_n$ are $\mathbb{P}$-names in $V$, then
\begin{itemize}
    \item 
    \!
    $\begin{aligned}[t]
    & p \Vdash_{\mathbb{P}} \phi(\dot{x}_1, \dots, \dot{x}_n) \iff \\
    & \forall g \ ((g \text{ is } \mathbb{P} \text{-generic over } V \text{ and } p \in g) \\
    & \implies V[g] \models \phi(\dot{x}_1[g], \dots, \dot{x}_n[g])), \text{ and}
    \end{aligned}$
    \item 
    \!
    $\begin{aligned}[t]
    \forall g \ ( & (g \text{ is } \mathbb{P} \text{-generic over } V \text{ and } V[g] \models \phi(\dot{x}_1[g], \dots, \dot{x}_n[g])) \\
    & \implies \exists q \ (q \Vdash_{\mathbb{P}} \phi(\dot{x}_1, \dots, \dot{x}_n) \text{ and } q \in g)) \text{.}
    \end{aligned}$
\end{itemize}
\end{thm}

Theorem \ref{thm310} intricately connects the forcing relation $\Vdash_{\mathbb{P}}^V$ with truth in $\mathbb{P}$-generic extensions and is fundamental to forcing as a technique. Colloquially known as the \emph{forcing theorem}, it enables us to reason about truth in generic extensions from within the ground model, and often reduces the argument from one about semantic entailment to one pertaining to combinatorial properties of partial orders. For more details about forcing and the proof of the forcing theorem, the reader is encouraged to read Chapter IV of \cite{kunen}. 

\begin{defi}
Define the relation $\leq_F$ on the set of CTMs as follows:
\begin{equation*}
    M \leq_F N \iff N \text{ is a forcing extension of } M \text{.}
\end{equation*}
\end{defi}

\begin{defi}
A \emph{(full) generic multiverse} is any set of CTMs closed under $\leq_F$.
\end{defi}

\begin{defi}
Let $V$ be a CTM. The \emph{(forcing) grounds of} $V$ is the set
\begin{equation*}
    \{W : V \text{ is a forcing extension of } W\} \text{.}
\end{equation*}
\end{defi}

\begin{defi}
Let $V$ be a CTM. The \emph{outward generic multiverse centred at} $V$ is the set
\begin{equation*}
    \mathbf{M}_F(V) := \{W : W \text{ is a forcing extension of } V\} \text{.}
\end{equation*}
\end{defi} 

The study of generic multiverses, including the coining of the term itself, arguably begins with Woodin in \cite{woodingen}. Since then, much has been studied about the structure of standard generic multiverses under $\leq_F$, with a particularly strong focus on the forcing grounds of fixed CTMs. On the other hand, there has been less interest in the structure of outward generic multiverses under $\leq_F$, perhaps due to the dearth of low-hanging fruits --- a large part of what is known about this structure are essentially theorems about forcing in the traditional sense.

A careful reader might have noticed the overloading of the notation $\cdot [\cdot]$ to represent both small extensions and forcing extensions. This is intentional, for the latter class is subsumed under the former.

\begin{fact}
For some $\mathbb{P} \in V$, let $g$ be a $\mathbb{P}$-generic filter over a CTM $V$. Then $V[g]$ is the smallest CTM $W$ for which
\begin{itemize}
    \item $V \subset W$, and
    \item $g \in W$.
\end{itemize}
As a consequence, $\mathbf{M}_F(V) \subset \mathbf{M}_S(V)$.
\end{fact}

It turns out that forcing extensions of a CTM $V$ are downward-closed in $\mathbf{M}(V)$ (and thus, also in $\mathbf{M}_S(V)$). This is just a rephrasing of the fact below.

\begin{fact}\label{fact221}
Let $V \subset U \subset W$ be CTMs such that
\begin{itemize}
    \item $W$ is a forcing extension of $V$, 
    \item $U$ is an outer model of $V$, and
    \item $U$ is an inner model of $W$.
\end{itemize}
Then 
\begin{itemize}
    \item $U$ is a forcing extension of $V$, and
    \item $W$ is a forcing extension of $U$.
\end{itemize}
\end{fact}

An immediate follow-up question to the previous two facts is, 
\begin{quote}
    ``Must $\mathbf{M}_F(V)$ always equal $\mathbf{M}_S(V)$?''
\end{quote}
There is an easy argument for the answer being ``no'', if we assume a sufficiently strong large cardinal axiom in addition to $\mathsf{ZFC}$.

\begin{prop}\label{prop318}
Let 
\begin{equation*}
    \mathsf{T} := \mathsf{ZFC} + \text{``}0^{\sharp}\text{ exists''.}
\end{equation*} 
Assume $\mathsf{ZFC} + \text{``there is a transitive model of } T \text{''}$. Then $\mathbf{M}_F(V) \subsetneq \mathbf{M}_S(V)$ for some CTM $V$.
\end{prop}

\begin{proof}
Given the hypothesis of the proposition, there is a CTM $W$ satisfying $\mathsf{T}$. Define 
\begin{gather*}
    V := L^W \\
    U := L[I]^W \text{,}
\end{gather*}
where $I$ is an uncountable set of Silver indiscernibles in $W$ witnessing the fact that $0^{\sharp}$ exists. Then $U$ is a small extension of $V$ generated by $I \in U$ but not a forcing extension of $V$.
\end{proof}

By Proposition \ref{prop318}, it is consistent that $\mathbf{M}_S(V) \setminus \mathbf{M}_F(V)$ is non-trivial --- and includes at least a cone of $(\mathbf{M}_S(V), \subset)$ --- under strong enough assumptions. However, the small outward multiverse example exhibited by the proof of the proposition is undesirable because it is centred at a universe that is by many measures, ``too small'' (e.g. it has a trivial theory of constructibility degrees). 

A much stronger and much more useful statement would be
\begin{equation}\label{eq31}
    \text{``} \mathbf{M}_F(V) \neq \mathbf{M}_S(V) \text{ for all } V \text{''.}
\end{equation}
Let us sketch how this can be true. We start with a universe $V$, force (with a proper class forcing notion) to an outer model $V[G]$ of $V$ satisfying ``$V[G]$ is not a set forcing extension of $V$'', then apply Fact \ref{fact37} to $V[G]$. The end result is an outer model $W$ of $V[G]$ such that $W = L^V[r]$ for some real $r \in W$. This can even be arranged such that $V$ is a definable class in $W$. Now if $W$ is a set forcing extension of $V$, then so are all intermediate outer models of $V$, including $V[G]$. But we have just ensured $V[G]$ is not a set forcing extension of $V$. 

This argument, which includes a proof of Fact \ref{fact37}, can be formalised in a conservative second order extension of $\mathsf{ZFC}$ (the second-order portion is needed for proper class forcing), so our meta-theory suffices when $V$ is a CTM. We have thus established --- albeit sketchily so --- the following.

\begin{fact}\label{fact319}
(\ref{eq31}) holds.
\end{fact}

In actuality, we can switch $V$ for any $U \in \mathbf{M}_S(V)$ in the argument above and obtain a stronger conclusion.

\begin{fact}\label{fact229}
Given a CTM $V$,
\begin{equation*}
    (\mathbf{M}_S(V) \setminus \mathbf{M}_F(V), \subset)
\end{equation*}
is a cofinal subposet of $(\mathbf{M}_S(V), \subset)$.
\end{fact}

Intuitively, Facts \ref{fact221} and \ref{fact229} tell us that there are many objects inaccessible by forcing. Do these objects have ``local first-order properties'' not shared by any set in any forcing extension? Much of the rest of this paper aims for a partial answer to the aforementioned question. 

\subsection{Theories with Constraints in Interpretation}

\emph{Theories with constraints in interpretation} (henceforth, \emph{TCI}s) were conceived in \cite{myself} as a convenient means of --- among other things --- looking at generic objects produced by set-theoretic forcing.

\begin{defi}(Lau, \cite{myself})
A \emph{first-order theory with constraints in interpretation} (\emph{first-order TCI}) --- henceforth, just \emph{theory with constraints in interpretation} (\emph{TCI}) --- is a tuple $(T, \sigma, \dot{\mathcal{U}}, \vartheta)$, where
\begin{itemize}
    \item $T$ is a first order theory with signature $\sigma$,
    \item $\dot{\mathcal{U}}$ is a unary relation symbol not in $\sigma$,
    \item $\vartheta$ is a function (the \emph{interpretation constraint map}) with domain $\sigma \cup \{\dot{\mathcal{U}}\}$, 
    \item if $x \in ran(\vartheta)$, then there is $y$ such that 
    \begin{itemize}[label=$\circ$]
        \item either $x = (y, 0)$ or $x = (y, 1)$, and
        \item if $\vartheta(\dot{\mathcal{U}}) = (z, a)$, then $y \subset z^n$ for some $n < \omega$, and
    \end{itemize}
    \item if $\vartheta(\dot{\mathcal{U}}) = (z, a)$, then 
    \begin{itemize}[label=$\circ$]
        \item $z \cap z^n = \emptyset$ whenever $1 < n < \omega$, and
        \item $z^m \cap z^n = \emptyset$ whenever $1 < m < n < \omega$.
    \end{itemize}
\end{itemize}
We call members of the interpretation constraint map \emph{interpretation constraints}.

For simplicity's sake, we always assume members of $T$ are in prenex normal form.
\end{defi}

\begin{defi}(Lau, \cite{myself})
Let $(T, \sigma, \dot{\mathcal{U}}, \vartheta)$ be a TCI. We say $$\mathcal{M} := (U; \mathcal{I}) \models^* (T, \sigma, \dot{\mathcal{U}}, \vartheta)$$ --- or $\mathcal{M}$ \emph{models} $(T, \sigma, \dot{\mathcal{U}}, \vartheta)$ --- iff all of the following holds:
\begin{itemize}
    \item $\mathcal{M}$ is a structure,
    \item $\sigma$ is the signature of $\mathcal{M}$,
    \item $\mathcal{M} \models T$,
    \item if $\vartheta(\dot{\mathcal{U}}) = (y, 0)$, then $U \subset y$,
    \item if $\vartheta(\dot{\mathcal{U}}) = (y, 1)$, then $U = y$, and
    \item for $\dot{X} \in \sigma$,
    \begin{itemize}[label=$\circ$]
        \item if $\dot{X}$ is a constant symbol and $\vartheta(\dot{X}) = (y, z)$, then $\mathcal{I}(\dot{X}) \in y \cap U$,
        \item if $\dot{X}$ is an $n$-ary relation symbol and $\vartheta(\dot{X}) = (y, 0)$, then $\mathcal{I}(\dot{X}) \subset y \cap U^{n}$,
        \item if $\dot{X}$ is an $n$-ary relation symbol and $\vartheta(\dot{X}) = (y, 1)$, then $\mathcal{I}(\dot{X}) = y \cap U^{n}$,
        \item if $\dot{X}$ is an $n$-ary function symbol and $\vartheta(\dot{X}) = (y, 0)$, then $$\{z \in U^{n+1} : \mathcal{I}(\dot{X})(z \! \restriction_n) = z(n)\} \subset y \cap U^{n+1}, \text{ and}$$
        \item if $\dot{X}$ is an $n$-ary function symbol and $\vartheta(\dot{X}) = (y, 1)$, then $$\{z \in U^{n+1} : \mathcal{I}(\dot{X})(z \! \restriction_n) = z(n)\} = y \cap U^{n+1}.$$
    \end{itemize}
\end{itemize}
We say $(T, \sigma, \dot{\mathcal{U}}, \vartheta)$ has a model if there exists $\mathcal{M}$ for which $\mathcal{M} \models^* (T, \sigma, \dot{\mathcal{U}}, \vartheta)$.
\end{defi}  

We can define a notion of complexity on TCIs. This will help us subsequently classify method definitions according to their complexity. 

\begin{defi}
Let $\phi$ be a first-order formula over a signature $\sigma$. We inductively define what it means for $\phi$ to be $\Pi_n$ or $\Sigma_n$ as $n$ ranges over the natural numbers.
\begin{enumerate}[label=(\arabic*)]
    \item If $n = 0$, then $\phi$ is $\Pi_n$ iff $\phi$ is $\Sigma_n$ iff $\phi$ is quantifier-free.
    \item\label{3292} If $n = m + 1$ for some $m < \omega$, then 
    \begin{enumerate}[label=(\alph*)]
        \item\label{3292a} $\phi$ is $\Pi_n$ iff there is a $\Sigma_m$ formula $\varphi$, a number $k < \omega$, and variable symbols $x_1, \dots, x_k$ not bound in $\varphi$, such that 
        \begin{equation*}
            \phi = \ulcorner \forall x_1 \dots \forall x_k \ \varphi \urcorner \text{, and}
        \end{equation*}
        \item\label{3292b} $\phi$ is $\Sigma_n$ iff there is a $\Pi_m$ formula $\varphi$, a number $k < \omega$, and variable symbols $x_1, \dots, x_k$ not bound in $\varphi$, such that 
        \begin{equation*}
            \phi = \ulcorner \exists x_1 \dots \exists x_k \ \varphi \urcorner \text{.}
        \end{equation*}
    \end{enumerate}
\end{enumerate}
Note that if $k = 0$ in \ref{3292}\ref{3292a} and \ref{3292}\ref{3292b}, then $\phi$ is $\Sigma_m$ and $\Pi_m$ respectively.
\end{defi}

\begin{defi}[Lau, \cite{myself}]
A TCI $(T, \sigma, \dot{\mathcal{U}}, \vartheta)$ is $\Pi_n$ iff $T$ contains only $\Pi_n$ sentences.

A TCI $(T, \sigma, \dot{\mathcal{U}}, \vartheta)$ is $\Sigma_n$ iff $T$ contains only $\Sigma_n$ sentences.
\end{defi}

As we shall see in Proposition \ref{prop323}, the existence of models of a TCI need not be absolute between between transitive models of $\mathsf{ZFC}$. There is thus a fundamental difference between model existence of TCIs and that of first-order theories. This should reflect in our definition of what it means for a TCI to be consistent.

\begin{defi}[Lau, \cite{myself}]
A TCI $(T, \sigma, \dot{\mathcal{U}}, \vartheta)$ is \emph{consistent} iff $(T, \sigma, \dot{\mathcal{U}}, \vartheta)$ has a model in some outer model of $V$. 
\end{defi}

It might seem at first glance, that the the consistency of a TCI is not a first-order property in the language of set theory, since it involves quantifying over all outer models. This is not a real problem because said definition is semantically equivalent to a first-order sentence in $V$ with parameters in $V$.

\begin{lem}\label{lem329}
Let $\mathfrak{T} = (T, \sigma, \dot{\mathcal{U}}, \vartheta)$ be a TCI. Then $\mathfrak{T}$ is consistent iff $$\Vdash_{Col(\omega, \lambda)} \exists \mathcal{M} \ (``\mathcal{M} \models^* \mathfrak{T}"),$$ where $\lambda \geq |H(|trcl(\mathfrak{T})|^+)|$.
\end{lem}

\begin{proof}
This is Lemma 5.25 of \cite{myself}.
\end{proof}

Lemma \ref{lem329} gives us a uniform way of checking in $V$ if a TCI is consistent, by appealing to a suitable forcing relation. 

Morally speaking, the consistency of a theory --- however it is defined --- should be absolute in a strong enough sense. This is the case for first-order theories, any of which consistency is absolute for transitive models of set theory. The following Lemma establishes a similar absoluteness property with regards to the consistency of a TCI.

\begin{lem}[Lau, \cite{myself}]\label{lem330}
Let $\mathfrak{T} = (T, \sigma, \dot{\mathcal{U}}, \vartheta)$ be a TCI. Then $\mathfrak{T}$ being consistent is absolute for transitive models of $\mathsf{ZFC}$ sharing the same ordinals.
\end{lem}

\begin{proof}
Let $V'$ and $W$ be transitive models of $\mathsf{ZFC}$ with $ORD^{V'} = ORD^W$ and $\mathfrak{T} \in V' \subset W$. If $\mathfrak{T}$ is consistent in $W$, then $\mathfrak{T}$ has a model in some outer model of $W$. Said outer model is also an outer model of $V'$, so $\mathfrak{T}$ is consistent in $V'$ as well.

Now assume $\mathfrak{T}$ is consistent in $V'$. Letting 
\begin{equation*}
    \lambda := |H(|trcl(\mathfrak{T})|^+)|
\end{equation*} 
evaluated in $V'$, Lemma \ref{lem329} gives us 
\begin{equation*}
    \Vdash_{Col(\omega, \lambda)} \exists \mathcal{M} \ (``\mathcal{M} \models^* \mathfrak{T}")
\end{equation*} 
in $V'$. Note that 
\begin{equation*}
    \mathbb{P} := Col(\omega, \lambda)^{V'}
\end{equation*} 
remains a forcing notion in $W$, so consider $g$ a $\mathbb{P}$-generic filter over $W$. Necessarily, $g$ is also $\mathbb{P}$-generic over $V'$, and further, $V'[g] \subset W[g]$. In $V'[g]$, $\mathfrak{T}$ is forced to have a model --- call it $\mathcal{M}$. Being a model of $\mathfrak{T}$ is absolute for transitive models of $\mathsf{ZFC}$, so $\mathcal{M} \models^* \mathfrak{T}$ holds in $W[g]$ too. Since $W[g]$ is an outer model of $W$, $\mathfrak{T}$ must be consistent in $W$.
\end{proof}

\section{Local Method Definitions}

If we look at forcing as a way to uniformly describe by intension members of a subset of $\mathbf{N}(V)$ for some CTM $V$, where the evaluation of intension is done outside $V$, then we quickly realise that said description can be very simple. We essentially just 
\begin{itemize}
    \item shortlist a class of structures in $V$ --- forcing notions augmented with predicates for their dense subsets --- and
    \item describe how we pick substructures --- generic filters --- of these structures outside $V$. 
\end{itemize} 
Note also that the description of each substructure $\mathfrak{A}'$ involves only information about its associated superstructure $\mathfrak{A}$, so we expect every reasonable universe containing $\mathfrak{A}$ to see that $\mathfrak{A}'$ fits the description. This is analogous to (albeit stronger than) the kind of local definablity one would expect the state transition function of a typical machine model of computation to satisfy.

\subsection{Locally Definable Methods of Small Extensions}

This idea of generating new structures outside the universe based locally on recipes defined in the universe can be formalised rather conveniently using the language of TCIs. 

\begin{defi}\label{def320}
Let $V$ be a CTM. A set $X \subset V$ is a $\mathbf{M}_S(V)$ \emph{method definition of small extensions} (henceforth, $\mathbf{M}_S(V)$ \emph{method definition}) iff it is non-empty and contains only TCIs.
\end{defi}

\begin{defi}\label{def323}
If $V$ is a CTM and $\mathfrak{T} \in V$ is a TCI, then the \emph{evaluation of} $\mathfrak{T}$, denoted $\mathrm{Eval}^V(\mathfrak{T})$, is the set
\begin{equation*}
    \{V[\mathcal{M}] : \exists W \ \exists \mathcal{M} \ (W \in \mathbf{M}(V) \wedge \mathcal{M} \in W \wedge \mathcal{M} \models^* \mathfrak{T})\} \text{.}
\end{equation*}
\end{defi}

\begin{defi}
Let $V$ be a CTM. A set $X \subset V$ is a $\mathbf{M}_S(V)$ \emph{local method definition of small extensions} (henceforth, $\mathbf{M}_S(V)$ \emph{local method definition}) iff it is a $\mathbf{M}_S(V)$ method definition and $X$ is definable (possibly as a proper class) in $V$ with parameters in $V$. 
\end{defi}

Just like how a formula in the language of arithmetic can be used to pick out a Turing degree, we want to have a TCI in $V$ pick out a degree of small extension of $V$. However, it would be absurd to expect any such TCI to isolate a single such degree, due to the non-constructive nature of outer models. As a result, there is no avoiding fuzziness when we evaluate TCIs to obtain degrees. This is the intuition behind $\mathrm{Eval}^V$ naturally taking TCIs to sets, instead of taking TCIs to points.

It might not be immediately obvious that whenever there are $W$ and $\mathcal{M}$ for which $\mathcal{M} \in W \in \mathbf{M}(V)$ and $\mathcal{M} \models^* \mathfrak{T}$ for some TCI $\mathfrak{T} \in V$, $V[\mathcal{M}]$ must exist. The next proposition shows that one can translate between models of a TCI $\mathfrak{T}$ and subsets of a set associated with $\mathfrak{T}$, in an absolute and uniform manner.

\begin{prop}[Lau, \cite{myself}]\label{prop322}
There are formulas $\phi$ and $\psi$ in the language of set theory with the following properties: 
\begin{enumerate}[label=(\arabic*)]
    \item $\phi$ and $\psi$ have two and three free variables respectively,
    \item $\phi$ defines a function from the class of all TCIs into the universe,
    \item $\psi$ defines a function from the class
    \begin{equation*}
        \{(\mathfrak{T}, \mathcal{M}) : \mathfrak{T} \text{ is a TCI and } \mathcal{M} \models^* \mathfrak{T}\}
    \end{equation*}
    into the universe,
    \item $\phi$ and $\psi$ are absolute for transitive models of $\mathsf{ZFC}$,
    \item\label{3245} whenever $\mathfrak{T}$ is a TCI, the relation
    \begin{equation*}
        R_{\mathfrak{T}} := \{(\mathcal{M}, x) : \psi(\mathfrak{T}, \mathcal{M}, x)\}
    \end{equation*}
    is one-one, 
    \item\label{3246} for all $\mathfrak{T}$, $\mathcal{M}$ and $x$, $\psi(\mathfrak{T}, \mathcal{M}, x)$ implies there is $\mathcal{L}$ for which $\phi(\mathfrak{T}, \mathcal{L})$ and $x \subset \mathcal{L}$, and
\end{enumerate}
\end{prop}

\begin{proof}
Let $\mathfrak{T}$ be a TCI. Using only information about $\mathfrak{T}$, we will constructively define a set $\mathcal{L}$. Set 
\begin{align*}
    \sigma' := \ & \sigma \cup \{\dot{\mathcal{U}}\} \text{, and} \\
    U := \ & \text{the unique } y \text{ for which there exists } z \text{ such that } \vartheta(\dot{\mathcal{U}}) = (y, z).
\end{align*}
For $\dot{X} \in \sigma'$, define $\mathcal{L}(\dot{X})$ as follows:
\begin{itemize}
    \item if $\dot{X}$ is a constant symbol and $\vartheta(\dot{X}) = (y, z)$, then $$\mathcal{L}(\dot{X}) := \{\ulcorner \dot{X} = x \urcorner : x \in y \cap U\},$$
    \item if $\dot{X}$ is an $n$-ary relation symbol and $\vartheta(\dot{X}) = (y, z)$, then $$\mathcal{L}(\dot{X}) := \{\ulcorner \dot{X}(x) \urcorner : x \in y \cap U^n\}, \text{ and}$$
    \item if $\dot{X}$ is an $n$-ary function symbol and $\vartheta(\dot{X}) = (y, z)$, then $$\mathcal{L}(\dot{X}) := \{\ulcorner \dot{X}(x \! \restriction_n) = x(n) \urcorner : x \in y \cap U^{n+1}\}.$$
\end{itemize}
Then 
\begin{align*}
    \mathcal{L}' := \ & \bigcup \{\mathcal{L}(\dot{X}) : \dot{X} \in \sigma'\} \text{, and} \\
    \mathcal{L} := \ & \text{the closure of }  \mathcal{L}' \text{ under negation.}
\end{align*}
This construction is $\Delta_0$-definable in the language of set theory with a single parameter, $\mathfrak{T}$, so it is absolute for transitive models of $\mathsf{ZFC}$. Use $\phi$ to denote this way of defining $\mathcal{L}$ from $\mathfrak{T}$.

Next let $\mathcal{M} = (M; \mathcal{I})$ be a model of $\mathfrak{T}$. Set 
\begin{equation*}
    U(\mathcal{M}) := \{\ulcorner \dot{\mathcal{U}}(x) \urcorner : x \in M\} \cup \{\ulcorner \neg \ \dot{\mathcal{U}}(x) \urcorner : x \in U \setminus M\} \text{.}
\end{equation*} 
Now define $\psi$ as follows: 
\begin{align*}
    \psi(\mathfrak{T}, \mathcal{M}, \Sigma) \iff (&\mathfrak{T} \text{ is a TCI and }\mathcal{M} \models^* \mathfrak{T} \text{ and} \\
    & \Sigma = (U(\mathcal{M}) \cup Diag(\mathcal{M})) \cap \mathcal{L}_{\mathfrak{T}}) \text{,}
\end{align*} 
where $Diag(\mathcal{M})$ is the atomic diagram of $\mathcal{M}$ and $\mathcal{L}_{\mathfrak{T}}$ is the unique $\mathcal{L}$ for which $\phi(\mathfrak{T}, \mathcal{L})$. Verily, $\psi$ is a $\Delta_1$ formula, because the binary relation $\models^*$ is $\Delta_1$-definable and the set $\mathcal{L}_{\mathfrak{T}}$ is $\Delta_1$-definable with parameter $\mathfrak{T}$. As such, $\psi$ must be absolute for transitive models of $\mathsf{ZFC}$. That \ref{3245} holds is straightforward based on the definition of $\psi$. 
\end{proof}

In the presence of Fact \ref{fact34}, Proposition \ref{prop322} gives validity to Definition \ref{def323}: the function $\mathrm{Eval}^V$ taking TCIs in $V$ to subsets of $\mathbf{M}_S(V)$ actually exists. Fix formulas $\phi$ and $\psi$ as in Proposition \ref{prop322}. Let 
\begin{align*}
    \mathcal{L}_{\mathfrak{T}} := \ & \text{ the unique } \mathcal{L} \text{ for which } \phi(\mathfrak{T}, \mathcal{L}) \text{, and} \\
    \Sigma(\mathfrak{T}, \mathcal{M}) := \ & \text{ the unique } x \text{ for which } \psi(\mathfrak{T}, \mathcal{M}, x) \text{.}
\end{align*}
Given a TCI $\mathfrak{T} = (T, \sigma, \dot{\mathcal{U}}, \vartheta)$ and sets $y, z$ with $\vartheta(\dot{\mathcal{U}}) = (y, z)$, we should think of $\mathcal{L}_{\mathfrak{T}}$ as containing all the possible atomic sentences over $\sigma$ that involve members of $y$ as parameters. Then for any model $\mathcal{M}$ of $\mathfrak{T}$, $\Sigma(\mathfrak{T}, \mathcal{M})$ ($\subset \mathcal{L}_{\mathfrak{T}}$, by \ref{3246} of Proposition \ref{prop322}) can be thought of as the ``$\mathfrak{T}$-specific atomic diagram''  of $\mathcal{M}$. According to \ref{3245} of Proposition \ref{prop322}, we can always recover $\mathcal{M}$ in a transitive model of $\mathsf{ZFC}$ from $\mathfrak{T}$ and $\Sigma(\mathfrak{T}, \mathcal{M})$ alone. As a consequence, 
\begin{equation*}
    V[\mathcal{M}] = V[x] \text{ if } V \text{ is a CTM and } \psi(\mathfrak{T}, \mathcal{M}, x) \text{ for some TCI } \mathfrak{T} \in V \text{.}
\end{equation*}

Observe that we can ``cover'' the entire small outward multiverse with a local method definition.

\begin{prop}\label{prop326}
Let $V$ be a CTM. There is a $\mathbf{M}_S(V)$ local method definition $X$ containing only $\Pi_0$ ($= \Sigma_0$) TCIs such that 
\begin{equation*}
    \bigcup \ ((\mathrm{Eval}^V)" X) = \mathbf{M}_S(V) \text{.}
\end{equation*}
\end{prop}

\begin{proof}
Fix a distinguished unary relation symbol $\dot{\mathcal{U}}$ and let $s \in V$. Define 
\begin{gather*}
    \vartheta_s := \{(\dot{\mathcal{U}}, (s, 0))\} \\
    \mathfrak{T}_s := (\emptyset, \emptyset, \dot{\mathcal{U}}, \vartheta_s) \text{.}
\end{gather*}
Then $\mathfrak{T}_s$ is a $\Pi_0$ TCI, and its models in any outer model $W$ of $V$ are exactly the subsets of $s$ in $W$. We are done by Fact \ref{fact34} if we set $X := \{\mathfrak{T}_s : s \in V\}$.
\end{proof}

To kickstart our journey of relating set-theoretic forcing to TCIs, let us introduce the notion of a generic model of a TCI. For future referencing in service of the ease of expression and reading, we define the abbreviation
\begin{equation*}
    \mathfrak{A}_{\mathfrak{T}} := (H(|trcl(\mathfrak{T})|^+); \in) \text{.}
\end{equation*}

\begin{ob}\label{ob331}
It is easily verifiable from the definitions of $\mathcal{L}_{\mathfrak{T}}$ and $\mathfrak{A}_{\mathfrak{T}}$ that $\mathcal{L}_{\mathfrak{T}} \in \mathfrak{A}_{\mathfrak{T}}$, so we have $[\mathcal{L}_{\mathfrak{T}}]^{< \omega} \in \mathfrak{A}_{\mathfrak{T}}$ as well.
\end{ob}

\begin{defi}
A triple $(\mathfrak{T}, \mathfrak{A}, \mathbb{P})$ is \emph{generically sensible} iff 
\begin{itemize}
    \item $\mathfrak{T}$ is a TCI,
    \item $\mathfrak{A} = (A; \in, \Vec{R})$ is a structure expanding on a transitive model of $\mathsf{ZFC - Powerset}$,
    \item $\mathbb{P}$ is a forcing notion, and
    \item $\{\mathfrak{T}, \mathbb{P}\} \subset A$.
\end{itemize}
\end{defi}

It is clear that $(\mathfrak{T}, \mathfrak{A}_{\mathfrak{T}}, \mathbb{P})$ is a generically sensible triple.

\begin{defi}[Lau, \cite{myself}]
Given a generically sensible triple $(\mathfrak{T}, \mathfrak{A}, \mathbb{P})$, a $(\mathbb{P}, \mathfrak{A})$\emph{-generic model} of $\mathfrak{T}$ is a model $\mathcal{M}$ of $\mathfrak{T}$ satisfying 
\begin{equation*}
    \Sigma(\mathfrak{T}, \mathcal{M}) = (\bigcup g) \cap \mathcal{L}_{\mathfrak{T}}
\end{equation*} 
for some $\mathbb{P}$-generic filter $g$ over $\mathfrak{A}$. In this case, we say $g$ \emph{witnesses} $\mathcal{M}$ \emph{is a} $(\mathbb{P}, \mathfrak{A})$-\emph{generic model of} $\mathfrak{T}$. We say $g$ \emph{witnesses a} $(\mathbb{P}, \mathfrak{A})$-\emph{generic model of} $\mathfrak{T}$ iff for some (necessarily unique) $\mathcal{M}$, $g$ witnesses $\mathcal{M}$ is a $(\mathbb{P}, \mathfrak{A})$-generic model of $\mathfrak{T}$.

We call $\mathcal{M}$ a $\mathfrak{A}$\emph{-generic model} of $\mathfrak{T}$ iff for some $\mathbb{P}$, $(\mathfrak{T}, \mathfrak{A}, \mathbb{P})$ is generically sensible and $\mathcal{M}$ is a $(\mathbb{P}, \mathfrak{A})$-generic model of $\mathfrak{T}$.

We call $\mathcal{M}$ a \emph{generic model} of $\mathfrak{T}$ iff for some $\mathfrak{A}$ and $\mathbb{P}$, $(\mathfrak{T}, \mathfrak{A}, \mathbb{P})$ is generically sensible and $\mathcal{M}$ is a $(\mathbb{P}, \mathfrak{A})$-generic model of $\mathfrak{T}$.   
\end{defi}

\begin{ob}[Lau, \cite{myself}]\label{smallvgen}
\leavevmode
\begin{enumerate}[label=(\arabic*)]
    \item\label{5282} If $g$ witnesses $\mathcal{M}$ is a $(\mathbb{P}, V)$-generic model of $\mathfrak{T}$, and $\bigcup g \subset \mathcal{L}_{\mathfrak{T}}$, then $V[g] = V[\mathcal{M}]$.
    \item\label{2svg} In the same vein as Observation \ref{ob0}, we see that given any consistent $\Pi_2$ TCI $\mathfrak{T}$, 
    \begin{align*}
        \forall x \ ( & x \text{ is a } (\mathbb{P}(\mathfrak{T}), \mathfrak{A}_{\mathfrak{T}}) \text{-generic model of } \mathfrak{T} \\
        & \iff x \text{ is a } (\mathbb{P}(\mathfrak{T}), V) \text{-generic model of } \mathfrak{T})
    \end{align*}
   in every outer model of $V$. As a result, we can safely talk about $(\mathbb{P}(\mathfrak{T}), V)$-generic models of $\mathfrak{T}$ without the need to quantify over all sets.
\end{enumerate}
\end{ob}

Our definition of a generic model is naturally arises from a nice characterisation of what we expect a generic model to be.

\begin{lem}\label{gmodelsinfe}
Let $\mathfrak{T}$ be a TCI. If $\mathcal{M}$ is a model of $\mathfrak{T}$ in some forcing extension of $V$, then $\mathcal{M}$ is a $V$-generic model of $\mathfrak{T}$.
\end{lem}

\begin{proof}
This is Lemma 5.30 of \cite{myself}.
\end{proof}

Lemma \ref{gmodelsinfe} actually tells us that $V$-generic models of a TCI $\mathfrak{T}$ are exactly the models of $\mathfrak{T}$ found in some forcing (i.e. generic) extension of $V$.

We end off this subsection by demonstrating that forcing can be regarded as a local method definition.

Unless stated otherwise, we work within a fixed CTM $V$ for the rest of this section, so that all mentions of $\mathbf{M}_S(V)$ in (local) method definitions can be conveniently suppressed.

\begin{prop}\label{prop323}
Let $\mathbb{P} = (P, \leq_{\mathbb{P}})$ be a partial order. Then there is a TCI $\mathfrak{T}(\mathbb{P}) = (T, \sigma, \dot{\mathcal{U}}, \vartheta)$ such that for a fixed unary relation symbol $\dot{X} \in \sigma$, if $\mathcal{M}$ is any set in an outer model of $V$, then
\begin{equation*}
    \mathcal{M} \models^* \mathfrak{T}(\mathbb{P}) \iff \{p : \mathcal{M} \models \dot{X}(p)\} \text{ is a } \mathbb{P} \text{-generic filter over } V \text{.}
\end{equation*}
\end{prop}

\begin{proof}
This is found in the proof of Theorem 5.34 of \cite{myself}, but we reproduce it the key portions here. 

Choose $\dot{\mathcal{U}}$, $\dot{\leq}$ and $\dot{G}$ and to be distinct relation symbols of arities $1$, $2$ and $1$ respectively. For each dense subset $D$ of $\mathbb{P}$, choose a fresh unary relation symbol $\dot{D}$. Set $\sigma$ to be $$\{\dot{\leq}, \dot{G}\} \cup \{\dot{D} : D \text{ is a dense subset of } \mathbb{P}\} \text{.}$$ We define $\vartheta$ on $\{\dot{\mathcal{U}}\} \cup \sigma$ as follows:
\begin{align*}
    \vartheta(\dot{\mathcal{U}}) := \ & (P, 1) \\
    \vartheta(\dot{\leq}) := \ & (\leq_{\mathbb{P}}, 1) \\
    \vartheta(\dot{G}) := \ & (P, 0) \\
    \vartheta(\dot{D}) := \ & (D, 1) \text{ for each dense subset } D \text{ of } \mathbb{P} \text{.}
\end{align*}
Now, have $T$ contain only the sentences
\begin{gather*}
    \ulcorner \forall p \ \forall q \ \exists r \ ((\dot{G}(p) \wedge \dot{G}(q)) \implies (\dot{G}(r) \wedge \dot{\leq}(r, p) \wedge \dot{\leq}(r, q))) \urcorner \text{,} \\
    \ulcorner \forall p \ \forall q \ ((\dot{\leq}(p, q) \wedge \dot{G}(p)) \implies \dot{G}(q)) \urcorner \text{, as well as all members of} \\
    \{\ulcorner \exists p \ (\dot{G}(p) \wedge \dot{D}(p)) \urcorner : D \text{ is a dense subset of } \mathbb{P}\} \text{.}
\end{gather*} 
Let $\mathfrak{T}(\mathbb{P}) := (T, \sigma, \dot{\mathcal{U}}, \vartheta)$. It is clear from our definition of $\mathfrak{T}(\mathbb{P})$ that whenever $\mathcal{M} \models^* \mathfrak{T}(\mathbb{P})$, the set $\{p : \mathcal{M} \models \dot{X}(p)\}$ is a $\mathbb{P}$-generic filter over $V$, and vice versa.
\end{proof}

\begin{cor}\label{cor324}
Forcing is expressible as a local method definition.
\end{cor}

\begin{proof}
The proof of Proposition \ref{prop323} is constructive, and can be made uniform across all possible $\mathbb{P}$ by choosing the symbols $\dot{\mathcal{U}}$, $\dot{\leq}$ and $\dot{G}$ in advance. This allows the function 
\begin{equation*}
    (F : \{\text{forcing notions}\} \longrightarrow V) [\mathbb{P} \mapsto \mathfrak{T}(\mathbb{P})]
\end{equation*}
to be definable in $V$. Obviously, $dom(F)$ is definable in $V$, so $ran(F)$ must be as well. 
\end{proof}

Let us choose a definable function $F$ as in the proof of Corollary \ref{cor324} and name it $\mathfrak{T}(\cdot)$, for later use and reference. Also, we shall use $\mathsf{Fg}$ to denote the local method definition of forcing; in other words,  $\mathsf{Fg} := ran(\mathfrak{T}(\cdot))$.

\subsection{Complexity of Local Method Definitions}\label{ss32}

What does it mean for a definition to be complex? Long, overwrought, convoluted. These are just some synonyms that may come to mind. in general, a more complex definition places more requirements on the object, or the class of objects, it defines. In the former scenario, it makes the object it defines \textit{a priori} less likely to exist; in the latter one, it defines a compratively smaller class of objects. Following this intuition, we formalise a way of comparing between two local method definitions.

\begin{defi}
Let $X, Y$ be local method definitions of a CTM $V$. We use 
\begin{enumerate}[label=(\arabic*)]
    \item $X \leq^M_w Y$ to denote the statement
    \begin{gather*}
        \text{``for each consistent } \mathfrak{T} \in X \text{ there is } \mathfrak{T}' \in Y \text{ such that} \\
        \emptyset \neq \mathrm{Eval}^V(\mathfrak{T}') \subset \mathrm{Eval}^V(\mathfrak{T}) \text{'',}
    \end{gather*}
    and
    \item $X \leq^M Y$ to denote the statement
    \begin{align*}
        \text{``} & \text{there is a function } F : X \longrightarrow Y \text{ definable in } V \text{ such that} \\ 
        & \emptyset \neq \mathrm{Eval}^V(F(\mathfrak{T})) \subset \mathrm{Eval}^V(\mathfrak{T}) \text{ for all consistent } \mathfrak{T} \in X \text{''.}
    \end{align*}
    When said statement is true, we say $F$ \emph{witnesses} $X \leq^M Y$.
\end{enumerate}
\end{defi}

\begin{defi}
Let $X, Y$ be local method definitions. We say
\begin{enumerate}[label=(\arabic*)]
    \item $X \equiv^M_w Y$ iff $X \leq^M_w Y$ and $Y \leq^M_w X$,
    \item $X \equiv^M Y$ iff $X \leq^M Y$ and $Y \leq^M X$,
    \item $X <^M_w Y$ iff $X \leq^M_w Y$ and $Y \not\leq^M_w X$, and
    \item $X <^M Y$ iff $X \leq^M Y$ and $Y \not\leq^M X$.
\end{enumerate}
\end{defi}

\begin{ob}
\leavevmode
\begin{enumerate}[label=(\arabic*)]
    \item $\leq^M_w$ and $\leq^M$ are transitive relations, so $\equiv^M_w$ and $\equiv^M$ are equivalence relations.
    \item $\leq^M_w$ and $\leq^M$, as subclasses of $V$, are only definable outside $V$, for their definitions require quantification over proper subclasses of $V$. 
    \item Obviously, $X \leq^M Y$ always implies $X \leq^M_w Y$, so $\leq^M_w$ is weaker than $\leq^M$. 
\end{enumerate}
\end{ob}

Intuitively, $Y$ is more complex than $X$ as a definition when $X \leq^M_w Y$ or $X \leq^M Y$, because $Y$ both refines and extends $X$. Refinement occurs because no matter set a description of $X$ picks out, $Y$ contains a description that picks out a smaller non-empty set. Extension occurs because $Y$ may have a description pick out a set that is not covered by any description of $X$. The difference between the two relations then boils down to whether a witness to said refinement and extension exists in $V$. Figure \ref{intuition} provides one way of visualising $X \leq^M Y$ with witness function $F$.

\begin{figure}[!ht]
    \centering
    \centerline{\hspace{1pt}\includegraphics[width=1.05\textwidth]{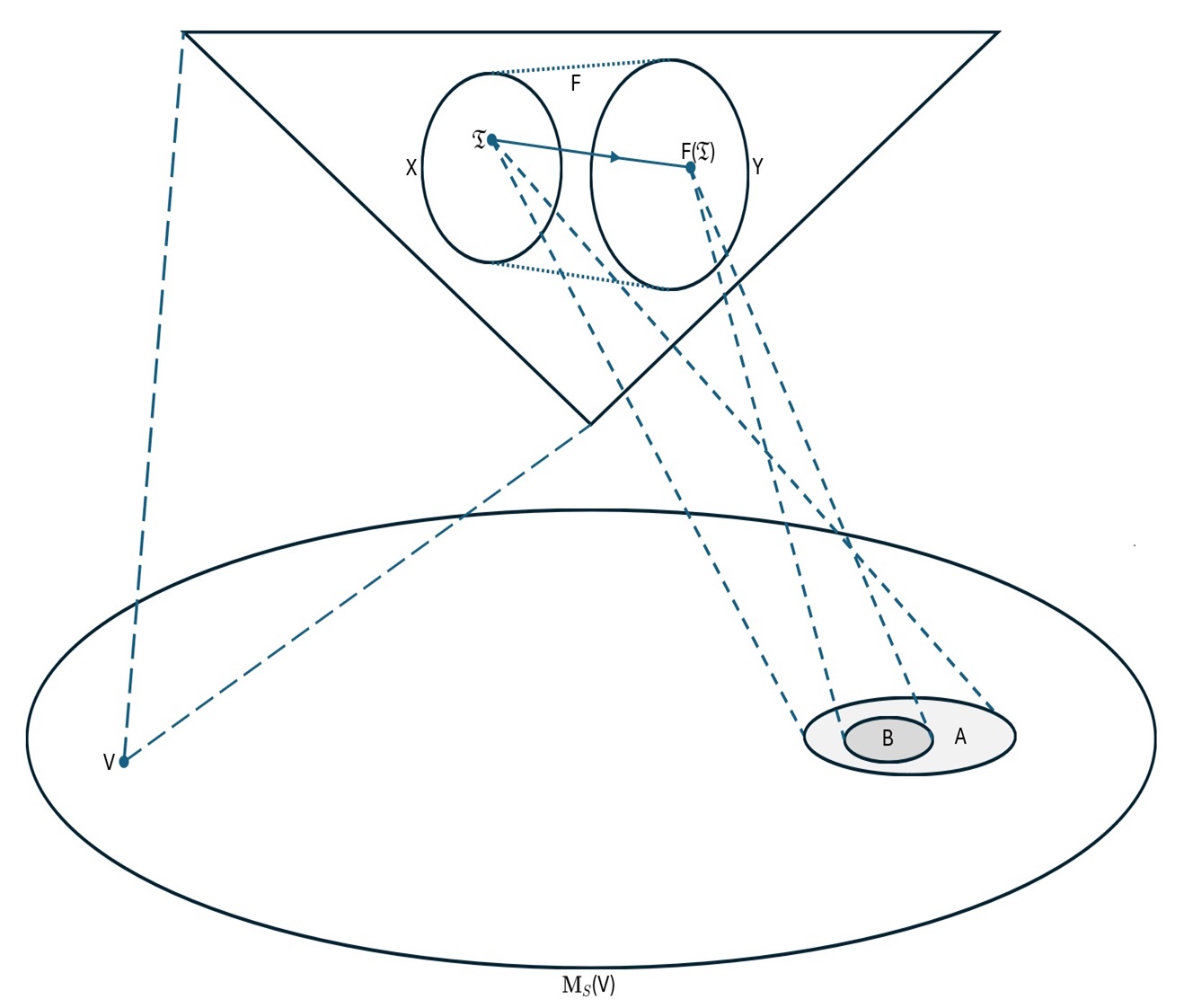}}
    \captionsetup{width=0.8\textwidth, format=hang}
    \caption[.]{Visual representation of a function $F$ witnessing $X \leq^M Y$, where $X$ and $Y$ are local method definitions. \\ Here, $\mathfrak{T}$ is an arbitrary consistent member of $X$, and \\ $B = \mathrm{Eval}^V(F(\mathfrak{T})) \subset \mathrm{Eval}^V(\mathfrak{T}) = A$.}
    \label{intuition}
\end{figure}

It would be good if $V$ can decide (albeit not uniformly) whether $X \leq^M Y$ for arbitrary local method definitions $X$ and $Y$. Unfortunately, there seems to be no straightforward indication of that: it is not clear if $V$ is always privy to proof of $X \leq^M Y$. For certain pairs $(X, Y)$ though, $X \leq^M Y$ is provable in $\mathsf{ZFC}$, and so $V$ must know it is true.

\begin{prop}\label{prop338}
Let $X, Y$ be local method definitions. If $X \subset Y$, then $X \leq^M Y$.
\end{prop}

\begin{proof}
The identity map on $X$ is definable in $V$ and witnesses $X \leq^M Y$.
\end{proof}

\begin{prop}
There is a greatest local method definition with respect to $\leq^M$.
\end{prop}

\begin{proof}
Let $Y$ be the class of all TCIs. Clearly $Y$ is a local method definition. Moreover, $X \subset Y$ for every local method definition $X$. By Proposition \ref{prop338}, $X \leq^M Y$ for every local method definition $X$.
\end{proof}

\begin{prop}
A local method definition is not smallest with respect to $\leq^M$ iff it contains a consistent TCI.
\end{prop}

\begin{proof}
Clearly, every local method definition containing no consistent TCIs is smallest with respect to $\leq^M$. For the converse, by Proposition \ref{prop326}, it suffices to show that for every $x \in V$, there is a small extension of $V$ not generated by a subset of $x$. But this is implied by the forcing notion $Col(|x|^+, |x|^+)^V$ adding no subsets of $x$ over $V$.
\end{proof}

We now define a natural hierarchy on the class of TCIs.

\begin{defi}\label{def339}
For $n < \omega$, we have the following local method definitions:
\begin{align*}
    \mathsf{\Pi^M_n} := \ & \{\mathfrak{T} : \mathfrak{T} \text{ is a } \Pi_n \text{ TCI}\}
    \\
    \mathsf{\Sigma^M_n} := \ & \{\mathfrak{T} : \mathfrak{T} \text{ is a } \Sigma_n \text{ TCI}\} \text{.}
\end{align*}
\end{defi}

\begin{prop}\label{prop340}
Let $n < \omega$. Then
\begin{itemize}
    \item $\mathsf{\Pi^M_n} \leq^M \mathsf{\Pi^M_{n+1}}$,
    \item $\mathsf{\Sigma^M_n} \leq^M \mathsf{\Sigma^M_{n+1}}$,
    \item $\mathsf{\Pi^M_n} \leq^M \mathsf{\Sigma^M_{n+1}}$, and
    \item $\mathsf{\Sigma^M_n} \leq^M \mathsf{\Pi^M_{n+1}}$.
\end{itemize}
\end{prop}

\begin{proof}
By Proposition \ref{prop338}.
\end{proof}

By Proposition \ref{prop340}, 
\begin{equation*}
    \{\mathsf{\Pi^M_n} : n < \omega\} \cup \{\mathsf{\Sigma^M_n} : n < \omega\}
\end{equation*}
forms a hierarchy of local method definitions with $\leq^M$-predecessor sets that grow with $n$. We shall call this the \emph{local method hierarchy}.

Mathematics and computer science are replete with hierarchies similar to the local method hierarchy, where syntactic forms of defining formulas are used to categorise sets. Examples include the projective, arithmetical and polynomial hierarchies. If we think of TCIs as augmentations of first-order theories with added constraints that are bounded but not first-order definable, then the local method hierarchy segregates TCIs based only on their first-order parts.

It turns out that most of the $\mathsf{\Pi^M_n}$'s are unnecessary in this hierarchy. 

\begin{lem}\label{prop252}
Let $n$ satisfy $1 \leq n < \omega$. For every $\mathfrak{T} \in \mathsf{\Pi^M_{n+1}}$ there are
\begin{itemize}
    \item $\mathfrak{T}' \in \mathsf{\Sigma^M_n}$, and
    \item a formula $\theta$ with two free variables,
\end{itemize}
such that 
\begin{itemize}
    \item $\theta$ is absolute for outer models of $V$, and
    \item in every outer model of $V$, $\theta$ defines a bijection from the set of models of $\mathfrak{T}$ into the set of models of $\mathfrak{T}'$.
\end{itemize}
\end{lem}

\begin{proof}
Let 
\begin{itemize}
    \item $1 \leq n < \omega$,
    \item $\mathfrak{T}  = (T, \sigma, \dot{\mathcal{U}}, \vartheta) \in \mathsf{\Pi^M_{n+1}}$, and 
    \item $\vartheta(\dot{\mathcal{U}}) = (y, z)$.
\end{itemize} 
We shall first construct $\mathfrak{T}' \in \mathsf{\Sigma^M_n}$ from $\mathfrak{T}$. Expand the signature $\sigma$ to $\sigma'$ by adding 
\begin{itemize}
    \item a unary relation symbol $\dot{T}$, as well as
    \item a constant symbol $\dot{c}$ for each $c \in y$, 
\end{itemize}
all of which are new to $\sigma$ and distinct from one another. Define $\vartheta'$ point-wise as follows:
\begin{gather*}
    \vartheta'(\dot{\mathcal{U}}) := (y, 1) \\
    \vartheta'(\dot{X}) := (y', 0) \text{ whenever } \dot{X} \in \sigma \text{ and } \vartheta(\dot{X}) = (y', z') \\
    \vartheta'(\dot{T}) := (y, z) \\
    \vartheta'(\dot{c}) := (\{c\}, 0) \text{ for all } c \in y \text{.}
\end{gather*}

Next, initialise $T^*$ to be 
\begin{align*}
    T & \cup \{\ulcorner \forall x_1 \dots \forall x_k \ \exists x_{k+1} \ (\dot{F}(x_1, \dots, x_k) = x_{k+1}) \urcorner \\
    & \mspace{25mu} : \dot{F} \in \sigma \text{ is a } k \text{-ary function symbol}\} \text{.}
\end{align*}
For each $\dot{X} \in \sigma$, so that $\vartheta(\dot{X})$ is of the form $(y', z')$, do the following.
\begin{enumerate}[label=Case (\arabic*)$_{\dot{X}}$:, leftmargin=70pt]
    \item $\dot{X}$ is a $k$-ary function symbol. Without loss of generality, we can assume $y' \subset y^{k+1}$. Add to $T^*$ every member of the set
    \begin{align*}
        \{ & \ulcorner \dot{X}(\dot{c}_1, \ldots, \dot{c}_k) = \dot{c}_{k+1} \implies \dot{T}(\dot{c}_1) \wedge \ldots \wedge \dot{T}(\dot{c}_{k+1}) \urcorner \\
        & : (c_1, \ldots, c_{k+1}) \in y'\} \text{.}
    \end{align*}
    If in addition, $z' = 1$, then also add to $T^*$ every member of the set
    \begin{align*}
        \{ & \ulcorner \dot{T}(\dot{c}_1) \wedge \ldots \wedge \dot{T}(\dot{c}_{k+1}) \implies \dot{X}(\dot{c}_1, \ldots, \dot{c}_k) = \dot{c}_{k+1} \urcorner \\
        & : (c_1, \ldots, c_{k+1}) \in y'\} \text{.}
    \end{align*}
    \item $\dot{X}$ is a $k$-ary relation symbol. Without loss of generality, we can assume $y' \subset y^k$. Add to $T^*$ every member of the set
    \begin{align*}
        \{ & \ulcorner \dot{X}(\dot{c}_1, \ldots, \dot{c}_k) \implies \dot{T}(\dot{c}_1) \wedge \ldots \wedge \dot{T}(\dot{c}_k) \urcorner \\
        & : (c_1, \ldots, c_k) \in y'\} \text{.}
    \end{align*}
    If in addition, $z' = 1$, then also add to $T^*$ every member of the set
    \begin{align*}
        \{ & \ulcorner \dot{T}(\dot{c}_1) \wedge \ldots \wedge \dot{T}(\dot{c}_k) \implies \dot{X}(\dot{c}_1, \ldots, \dot{c}_k) \urcorner \\
        & : (c_1, \ldots, c_k) \in y'\} \text{.}
    \end{align*}
    \item $\dot{X}$ is a constant symbol. Without loss of generality, we can assume $y' \subset y$. Add to $T^*$ the sentence
    \begin{align*}
        \ulcorner \dot{T}(\dot{X}) \urcorner \text{.}
    \end{align*}
\end{enumerate}
Now $T^*$, like $T$, contains only $\Pi_{n+1}$ sentences.

Fix any formula $\phi \in T^*$. Then $\phi$ is of the form 
\begin{equation*}
    \ulcorner \forall x_1 \dots \forall x_k \ \varphi \urcorner \text{,}
\end{equation*}
where $k < \omega$ and $\varphi$ is a $\Sigma_n$ formula of which leading quantifier --- should it exist --- is not a universal quantifier. Note that such $k$ and $\varphi$ are unique for $\phi$. If $\Vec{a} \in {^{k}{\{\dot{c} : c \in y\}}}$, use $\phi_{\Vec{a}}$ to denote the result of running the following procedure on $\phi$.
\begin{enumerate}[label=(\arabic*)]
    \item For each subformula $\psi$ containing at least one quantifier, in descending order of length (which is necessarily linear due to $\phi$ being in prenex normal form), do as per the cases below.
    \begin{enumerate}[label=Case (\arabic*)$_{\psi}$:, leftmargin=70pt]
        \item $\psi = \ulcorner \forall x \ \psi' \urcorner$ for some $x$ and $\psi'$. In this case, replace $\psi'$ with the string
        \begin{equation*}
            \ulcorner (\dot{T}(x) \implies \psi') \urcorner \text{.}
        \end{equation*}
        \item $\psi = \ulcorner \exists x \ \psi' \urcorner$ for some $x$ and $\psi'$. In this case, replace $\psi'$ with the string
        \begin{equation*}
            \ulcorner (\dot{T}(x) \wedge \psi') \urcorner \text{.}
        \end{equation*}
    \end{enumerate}
    \item For each $i$ such that $1 \leq i \leq k$, remove all instances of the string $\ulcorner \forall x_i \urcorner$.
    \item Substitute $\Vec{a}(i - 1)$ for every instance of $x_i$ whenever $1 \leq i \leq k$.
\end{enumerate}
It is not hard to verify that the aforementioned procedure is well-defined and produces a $\Sigma_n$ sentence over $\sigma'$. As a result, 
\begin{equation*}
    T_{\phi} := \{\phi_{\Vec{a}} : \Vec{a} \in {^{k}{\{\dot{c} : c \in y\}}}\}
\end{equation*}
is a set of $\Sigma_n$ sentences over $\sigma'$. We set 
\begin{equation*}
    T' := \bigcup \{T_{\phi} : \phi \in T^*\} \text{.}
\end{equation*}
so that $\mathfrak{T}' := (T', \sigma', \dot{\mathcal{U}}, \vartheta') \in \mathsf{\Sigma^M_n}$. Then the following hold true in every outer model of $V$.
\begin{enumerate}[label=(T\arabic*)]
    \item\label{316t1} Given any model $\mathcal{M}$ of $\mathfrak{T}'$,
    \begin{equation*}
        (\dot{T}^{\mathcal{M}}; \sigma^{\mathcal{M}})
    \end{equation*} 
    is a model of $\mathfrak{T}$. 
    \item\label{316t2} Every model of $\mathfrak{T}$ can be uniquely and constructively extended and expanded to a model of $\mathfrak{T}'$.
\end{enumerate}
It is clear that the transformations involved in \ref{316t1} and \ref{316t2} are absolute for outer models of $V$.
\end{proof}

An important upshot of Lemma \ref{prop252} is the corollary below.

\begin{cor}
$\mathsf{\Pi^M_{n+1}} \leq^M \mathsf{\Sigma^M_n}$ for all $n$ satisfying $1 \leq n < \omega$.
\end{cor}

We are interested in how $\mathsf{Fg}$ might fit into the local method hierarchy. To that end, let us first make a simple observation.

\begin{prop}\label{prop343}
$\mathsf{Fg} \leq^M \mathsf{\Sigma^M_1}$.
\end{prop}

\begin{proof}
For all forcing notions $\mathbb{P}$, the TCI $\mathfrak{T}(\mathbb{P})$ is always a member of $\mathsf{\Pi^M_2}$ by the proof of Proposition \ref{prop323}. That $\mathfrak{T}(\cdot)$ is definable in $V$ makes it a witness to $\mathsf{Fg} \leq^M \mathsf{\Pi^M_2}$. Lemma \ref{prop252} then implies $\mathsf{Fg} \leq^M \mathsf{\Sigma^M_1}$.
\end{proof}

\begin{lem}\label{lem344}
Let $\mathfrak{T} \in V$ be a consistent TCI and $\mathbb{P} = (P, \leq_{\mathbb{P}})$ be a forcing notion such that 
\begin{enumerate}[label=(\arabic*)]
    \item\label{3441} $\leq_{\mathbb{P}} \ = \ \supset \cap \ P$,
    \item\label{3442} every member of $P$ is a finite set, and
    \item\label{3443} $\Vdash_{\mathbb{P}}^V$``$\bigcup \dot{G} \subset \mathcal{L}_{\mathfrak{T}}$ and $\dot{G}$ witnesses a $(\mathbb{P}, \dot{V})$-generic model of $\mathfrak{T}$'',  
\end{enumerate}
where $\dot{G}$ and $\dot{V}$ are the canonical names for the generic filter on $\mathbb{P}$ and for the ground model, respectively. Then $\emptyset \neq \mathrm{Eval}^V(\mathfrak{T}(\mathbb{P})) \subset \mathrm{Eval}^V(\mathfrak{T})$.
\end{lem}

\begin{proof}
First, $\mathrm{Eval}^V(\mathfrak{T}(\mathbb{P}))$ is exactly the set of all $\mathbb{P}$-generic extensions of $V$, so $\emptyset \neq \mathrm{Eval}^V(\mathfrak{T}(\mathbb{P}))$. 

Let $U \in \mathrm{Eval}^V(\mathfrak{T}(\mathbb{P}))$, so that $U = V[g]$ for some $\mathbb{P}$-generic filter $g$ over $V$. By \ref{3443}, there is $\mathcal{M} \in U$ such that $\mathcal{M} \models^* \mathfrak{T}$, which implies $V[\mathcal{M}] \subset U$ and $V[\mathcal{M}] \in \mathrm{Eval}^V(\mathfrak{T})$. That 
\begin{equation*}
    \Vdash_{\mathbb{P}}^V \text{``} \bigcup \dot{G} \subset \mathcal{L}_{\mathfrak{T}} \text{''}
\end{equation*}
means $\bigcup g$ is definable from $\mathcal{M}$ in any transitive model of $\mathsf{ZFC}$, so $\bigcup g \in V[\mathcal{M}]$. To show $U \subset V[\mathcal{M}]$, it suffices to show $g$ is recoverable from $\bigcup g$ in $V[\mathcal{M}]$ with the help of parameters in $V$. 

We claim $g = [\bigcup g]^{< \omega} \cap P$. Clearly $g \subset [\bigcup g]^{< \omega} \cap P$ due to \ref{3442}. Next assume $p \in [\bigcup g]^{< \omega} \cap P$. Then for each $x \in p$, there must be some $q_x \in g$ containing $x$. As $p$ is a finite set and $g$ is a filter on $\mathbb{P}$, 
\begin{equation*}
    S := \{q_x : x \in p\}
\end{equation*}
has a common extension, say $q$, in $g$. By \ref{3441}, $p \subset \bigcup S \subset q$, so also $p \in g$. This proves our claim as well as the lemma.
\end{proof}

Lemma \ref{lem344} provides a direction towards proving $\mathsf{\Sigma^M_1} \leq^M \mathsf{Fg}$: we can try to define a function $F$ on $\mathsf{\Sigma^M_1}$ such that whenever $\mathfrak{T} \in \mathsf{\Sigma^M_1}$ is consistent, $F(\mathfrak{T}) = \mathfrak{T}(\mathbb{P})$ for some forcing notion $\mathbb{P}$ satisfying the hypothesis of Lemma \ref{lem344} in conjunction with $\mathfrak{T}$. 

Putting aside $\mathsf{Fg}$ for a moment, let us consider the local method hierarchy in and of itself. Notice that we have neither proven nor disproven anything about the size of
\begin{equation*}
    \{\mathsf{\Pi^M_n} : n < \omega\} \cup \{\mathsf{\Sigma^M_n} : n < \omega\} \text{ modulo } \equiv^M \text{,}
\end{equation*}
or equivalently, 
\begin{equation*}
    \{\mathsf{\Pi^M_1}\} \cup \{\mathsf{\Sigma^M_n} : n < \omega\} \text{ modulo } \equiv^M \text{,}
\end{equation*}
besides the obvious fact that it is countable and non-zero. Indeed, there seems to be no easy way of separating the rungs of the hierarchy as yet. This appears in stark contrast with the more renowned arithmetical and projective hierarchies, where separation happens ``everywhere''. However, by no means is this a reason to dismiss (our definition of) the hierarchy, or discourage the study thereof. One need not look far to find a well-studied hierarchy of the same ilk with the same ``issue'': the polynomial hierarchy, in which separation of any kind is equivalent to $\mathsf{P} \neq \mathsf{NP}$. 

\begin{ques}\label{q325}
\leavevmode
\begin{enumerate}[label=(\arabic*)]
    \item\label{325a} Are there $m, n < \omega$ for which $\mathsf{\Sigma^M_m} \not\equiv^M \mathsf{\Sigma^M_n}$?
    \item\label{325b} Is it true that $\mathsf{\Pi^M_1} \leq^M \mathsf{\Sigma^M_0}$?
    \item\label{325c} Is there a TCI $\mathfrak{T}$ such that $\{\mathfrak{T}\} \not\leq^M \mathsf{Fg}$?
\end{enumerate}
\end{ques}

So far, every small extension of $V$ we know how to construct either
\begin{itemize}
    \item is a forcing extension of $V$, or
    \item depends on information extracted from a proper-class-sized fragment of $V$, if not more.
\end{itemize} 
Naturally, we might wonder if it is possible to generate small extensions of $V$ based locally on a set in $V$, in a way that differs substantively from set forcing. This is essentially what \ref{325c} of Question \ref{q325} is asking. Should it be answered in the affirmative, non-trivial separation of the local method hierarchy becomes more hopeful, towards which \ref{325a} of Question \ref{q325} can be a guiding question. Should it be answered in the negative, $\mathsf{Fg}$ is a greatest local method definition with respect to $\leq^M$, and the local method hierarchy stops at $\mathsf{\Sigma^M_1}$ --- in this case \ref{325b} of Question \ref{q325} becomes the main question surrounding the relation $\leq^M$. 

\section{Categorising Forcing}

In this section, we complete what we started in Subsection \ref{ss32}, and associate set forcing with a rung of the local method hierarchy. Additionally, we will study different witnesses to the fact that $\mathsf{\Pi^M_2} \leq^M \mathsf{Fg}$.

\subsection{Forcing is \texorpdfstring{$\Sigma_1$}{Σ-1} (is \texorpdfstring{$\Pi_2$}{Π-2})}

This subsection is dedicated to showing $\mathsf{Fg} \equiv^M \mathsf{\Sigma^M_1}$. One direction of the proof is done in Proposition \ref{prop343}. For the other direction, we will identify a witness to $\mathsf{\Pi^M_2} \leq^M \mathsf{Fg}$. This witness can be defined without referencing any witness to $\mathsf{\Pi^M_2} \leq^M \mathsf{\Sigma^M_1}$.

\begin{defi}
Let $\mathfrak{T} \in V$ be a TCI. Define 
\begin{equation*}
    P(\mathfrak{T}) := \{p \in [\mathcal{L}_{\mathfrak{T}}]^{< \omega} : \exists W \! \in \! \mathbf{M}(V) \ \exists \mathcal{M} \ (\mathcal{M} \in W \text{, } \mathcal{M} \models^* \mathfrak{T} \text{ and } p \subset \Sigma(\mathfrak{T}, \mathcal{M}))\}
\end{equation*}
\end{defi}

It may not be clear that $P(\mathfrak{T})$ is a member of $V$ for arbitrary $\mathfrak{T} \in V$. We prove this in the next lemma.

\begin{defi}
Let $\mathfrak{T} \in V$ be a TCI. Define 
\begin{equation*}
    P'(\mathfrak{T}) := \{p \in [\mathcal{L}_{\mathfrak{T}}]^{< \omega} : \ \Vdash^V_{Col(w, |\mathfrak{A}_{\mathfrak{T}}|)} \exists \mathcal{M} \ (\text{``} \mathcal{M} \models^* \mathfrak{T} \text{ and } p \subset \Sigma(\mathfrak{T}, \mathcal{M}) \text{''})\}
\end{equation*}
\end{defi}

\begin{lem}
Let $\mathfrak{T} \in V$ be a TCI. Then $P(\mathfrak{T}) = P'(\mathfrak{T})$, so there is a definition of $P(\mathfrak{T})$ uniform over all TCIs $\mathfrak{T}$ in $V$.
\end{lem}

\begin{proof}
Noting that $|\mathfrak{A}_{\mathfrak{T}}| = |trcl(\mathfrak{A}_{\mathfrak{T}})|$, this is essentially the proof of Lemma 3.35 of \cite{myself} with different nomenclature.
\end{proof}

\begin{defi}
For each TCI $\mathfrak{T} \in V$, set 
\begin{equation*}
    \mathbb{P}(\mathfrak{T}) := (P(\mathfrak{T}), \supset \cap \ P(\mathfrak{T})) \text{.}
\end{equation*}
\end{defi}

\begin{rem}
Now that $\mathbb{P}(\mathfrak{T})$ has been defined for all TCIs $\mathfrak{T}$, it is clear from Lemma \ref{prop252} and Proposition \ref{prop343} that Question 5.70 of \cite{myself} can be answered in the affirmative. In other words, letting
\begin{equation*}
    \mathcal{C} := \{\mathfrak{T} : \mathfrak{T} \text{ is a } \Sigma_1 \text{ TCI}\} \text{,}
\end{equation*}
we can conclude the following:
\begin{enumerate}[label=(\alph*)]
    \item $\mathcal{C} \subsetneq \{\mathfrak{T} : \mathfrak{T} \text{ is a } \Pi_2 \text{ TCI}\}$, and
    \item for each forcing notion $\mathbb{P}$, there is $\mathfrak{T} \in \mathcal{C}$ for which
    \begin{itemize}
        \item $\mathbb{P}$ and $\mathbb{P}(\mathfrak{T})$ are forcing equivalent, and
        \item 
        \!
        $\begin{aligned}[t]
            & \{V[\mathcal{M}] : \mathcal{M} \models^* \mathfrak{T} \text{ in an outer model of } V\} \\
            & \mspace{70mu} = \{V[g] : g \text{ is } \mathbb{P} \text{-generic over } V\} \text{,}
        \end{aligned}$
    \end{itemize}
    so that also
    \begin{align*}
        & \{V[\mathcal{M}] : \mathcal{M} \models^* \mathfrak{T} \text{ in an outer model of } V\} \\
        & = \{V[g] : g \text{ is } \mathbb{P}(\mathfrak{T}) \text{-generic over } V\} \text{.}
    \end{align*}
\end{enumerate} 
\end{rem}

The definable function $\mathfrak{T} \mapsto \mathfrak{T}(\mathbb{P}(\mathfrak{T}))$, restricted to $\mathsf{\Pi^M_2}$, will be our witness to $\mathsf{\Pi^M_2} \leq^M \mathsf{Fg}$. A trivial observation is that \ref{3441} and \ref{3442} of Proposition \ref{lem344} hold for $\mathbb{P}(\mathfrak{T})$ and $P(\mathfrak{T})$ in place of $\mathbb{P}$ and $P$ respectively. Furthermore, it is always true that
\begin{equation*}
    \Vdash_{\mathbb{P}(\mathfrak{T})} \bigcup \dot{G} \subset \mathcal{L}_{\mathfrak{T}} \text{,}
\end{equation*}
so we are left to prove
\begin{equation}\label{32}
    \Vdash_{\mathbb{P}(\mathfrak{T})} \text{``} \dot{G} \text{ witnesses a } (\mathbb{P}(\mathfrak{T}), \dot{V}) \text{-generic model of } \mathfrak{T} \text{''}
\end{equation}
for every consistent $\Pi_2$ TCI $\mathfrak{T}$. In \cite{myself} this is done by appealing to the more general framework of forcing with language fragments. Fix a consistent $\Pi_2$ TCI $\mathfrak{T}$. Let us hereby briefly outline the proof of (\ref{32}).

In the aforementioned forcing framework, we allow potentially any set to be interpreted as a language, by extending the negation operator from classical first order logic to all sets. In other words, we define a canonical negation function $\neg$ on $V$ as follows: 
\begin{align*}
    \neg x := \neg(x) = 
    \begin{cases}
        y & \text{if } x = \ulcorner \neg y \urcorner \text{ for some } y \\
        \ulcorner \neg x \urcorner & \text{otherwise}.
    \end{cases}
\end{align*}
A set $\mathcal{L}$ is \emph{closed under negation} iff for each $\phi \in \mathcal{L}$, $\neg \phi \in \mathcal{L}$. The aim of the framework is to study certain definable subsets of a set closed under negation from the perspective of a larger structure. 

A structure $\mathfrak{A}$ is $\mathcal{L}$\emph{-suitable} iff it expands on a model of a sufficiently strong set theory ($\mathsf{ZFC -}$ $\mathsf{Powerset}$ is more than enough) and $\mathcal{L}$ is a definable class in $\mathfrak{A}$. We define the language $\mathcal{L}^*_{\mathfrak{A}}$ by enlarging the signature of $\mathfrak{A}$ with members of its base sets as constants and a fresh unary relation symbol $\ulcorner E \urcorner$. Morally, $\ulcorner E \urcorner$ is to be interpreted as a subset of $\mathcal{L}$ when $\mathfrak{A}$ is $\mathcal{L}$-suitable. $\mathcal{L}^*_{\mathfrak{A}}$ thus enables us to impose first-order requirements on subsets of $\mathcal{L}$. 

We say ``$\Sigma$ $\Gamma(\mathcal{L}, \mathfrak{A})$-certifies $p$'' iff 
\begin{itemize}
    \item $\mathcal{L}$ is closed under negation and $\mathfrak{A}$ is $\mathcal{L}$-suitable,
    \item $\mathfrak{A}$, augmented with the predicate $\Sigma$ that interprets $\ulcorner E \urcorner$, satisfies the theory $\Gamma \subset \mathcal{L}^*_{\mathfrak{A}}$, and
    \item $p \subset \Sigma$.
\end{itemize}
Syntactically, it makes sense to talk about $\Pi_n$ and $\Sigma_n$ formulas and sentences in $\mathcal{L}^*_{\mathfrak{A}}$, although these classes are defined a little differently from their counterparts in a typical set-theoretic context. Such a syntactic classification turns out to have interesting implications on forcing constructions.

There are a few lemmas, in varying degrees of generality, that connect genericity to relations akin to ``$\Gamma(\mathcal{L}, \mathfrak{A})$-certification''. The following is the most relevant to our intended use case.

\begin{lem}\label{lem349}
Let $W$, $\lambda$, $\mathfrak{A}$, $\mathcal{L}$, $\Gamma$, $P$, $\mathbb{P}$ and $g$ be such that
\begin{itemize}
    \item $\Gamma$ contains only $\Pi_2$ $\mathcal{L}^*_{\mathfrak{A}}$ sentences,
    \item $|trcl(\mathfrak{A})| \leq \lambda$,
    \item $P = \{p \in [\mathcal{L}]^{< \omega} : \ \Vdash_{Col(\omega, \lambda)} \exists \Sigma \ (``\Sigma \ \Gamma(\mathcal{L}, \mathfrak{A}) \text{-certifies } p")\} \neq \emptyset$,
    \item $\mathbb{P} = (P, \supset \cap \ P)$, 
    \item $\mathbb{P}$ is a definable class in $\mathfrak{A}$,
    \item $W$ is an outer model of $V$, and
    \item $g \in W$ is a $\mathbb{P}$-generic filter over $\mathfrak{A}$.
\end{itemize}
Then $\bigcup g$ $\Gamma(\mathcal{L}, \mathfrak{A})$-certifies $\emptyset$.

In particular, if $g$ is $\mathbb{P}$-generic over $V$, then $\bigcup g$ $\Gamma(\mathcal{L}, \mathfrak{A})$-certifies $\emptyset$ in $V[g] = V[\bigcup g]$.
\end{lem}

\begin{proof}
This is (implied by) Lemma 3.39 of \cite{myself}.
\end{proof}

The next lemma allows us to transform an arbitrary $\Pi_2$ TCIs into a form amenable with our forcing framework, so that Lemma \ref{lem349} can be applied.

\begin{lem}\label{lem350}
For each $\Pi_2$ TCI $\mathfrak{T}$ there is $\Gamma_{\mathfrak{T}}$ such that
\begin{enumerate}[label=(\arabic*)]
    \item\label{3551} $\Gamma_{\mathfrak{T}}$ contains only $\Pi_2$ $(\mathcal{L}_{\mathfrak{T}})^*_{\mathfrak{A}_{\mathfrak{T}}}$ sentences, and
    \item\label{3552} for every set $x$,
    \begin{equation*}
        \exists \mathcal{M} \ (\mathcal{M} \models^* \mathfrak{T} \text{ and } \Sigma(\mathfrak{T}, \mathcal{M}) = x) \iff x \ \Gamma_{\mathfrak{T}} (\mathcal{L}_{\mathfrak{T}}, \mathfrak{A}_{\mathfrak{T}})\text{-certifies } \emptyset \text{.}
    \end{equation*}
\end{enumerate}
\end{lem}

\begin{proof}
This is implied by Lemmas 5.17 and 5.22 of \cite{myself} (cf. Proposition \ref{prop322}).
\end{proof}

We can now derive (\ref{32}) from Lemmas \ref{lem349} and \ref{lem350}. 

\begin{thm}[Lau, \cite{myself}]\label{thm351}
Let $\mathfrak{T} = (T, \sigma, \dot{\mathcal{U}}, \vartheta)$ be a consistent $\Pi_2$ TCI. Then every $\mathbb{P}(\mathfrak{T})$-generic filter over $V$ witnesses a $(\mathbb{P}(\mathfrak{T}), V)$-generic model of $\mathfrak{T}$, or equivalently, (\ref{32}) holds.
\end{thm}

\begin{proof}
We first apply Lemma \ref{lem350} to obtain a $\Gamma_{\mathfrak{T}}$ such that
\begin{enumerate}[label=(\alph*)]
    \item\label{3521} $\Gamma_{\mathfrak{T}}$ contains only $\Pi_2$ $(\mathcal{L}_{\mathfrak{T}})^*_{\mathfrak{A}_{\mathfrak{T}}}$ sentences, and
    \item\label{3522} for every set $x$,
    \begin{equation*}
        \exists \mathcal{M} \ (\mathcal{M} \models^* \mathfrak{T} \text{ and } \Sigma(\mathfrak{T}, \mathcal{M}) = x) \iff x \ \Gamma_{\mathfrak{T}} (\mathcal{L}_{\mathfrak{T}}, \mathfrak{A}_{\mathfrak{T}})\text{-certifies } \emptyset \text{.}
    \end{equation*}
\end{enumerate}
Next, note that
\begin{enumerate}[label=(\alph*)]
    \setcounter{enumi}{2}
    \item\label{3523} Observation \ref{ob331} holds,
    \item $|\mathfrak{A}_{\mathfrak{T}}| = |trcl(\mathfrak{A}_{\mathfrak{T}})|$, and
    \item\label{3525} for all $x$ and $p$,
    \begin{equation*}
        x \ \Gamma_{\mathfrak{T}} (\mathcal{L}_{\mathfrak{T}}, \mathfrak{A}_{\mathfrak{T}})\text{-certifies } p \iff (x \ \Gamma_{\mathfrak{T}} (\mathcal{L}_{\mathfrak{T}}, \mathfrak{A}_{\mathfrak{T}})\text{-certifies } \emptyset \text{ and } p \subset x) \text{.}
    \end{equation*}
\end{enumerate}
Let $g$ be a $\mathbb{P}(\mathfrak{T})$-generic filter over $V$. Then the theorem follows directly from Lemma \ref{lem349}, as the hypotheses of said lemma are satisfied with
\begin{itemize}
    \item $V[g]$ in place of $W$,
    \item $|\mathfrak{A}_{\mathfrak{T}}|$ in place of $\lambda$,
    \item $\mathfrak{A}_{\mathfrak{T}}$ in place of $\mathfrak{A}$,
    \item $\mathcal{L}_{\mathfrak{T}}$ in place of $\mathcal{L}$,
    \item $\Gamma_{\mathfrak{T}}$ in place of $\Gamma$,
    \item $P(\mathfrak{T})$ in place of $P$, and
    \item $\mathbb{P}(\mathfrak{T})$ in place of $\mathbb{P}$,
\end{itemize}
bearing in mind \ref{3521} to \ref{3525}.
\end{proof}

It should be emphasised that the proof of Lemma \ref{lem350} provides a uniform way of constructing $\Gamma_{\mathfrak{T}}$ from any TCI $\mathfrak{T}$, such that \ref{3552} of the lemma is satisfied. If in addition, $\mathfrak{T}$ is $\Pi_2$, then the $\Gamma_{\mathfrak{T}}$ constructed also satisfies \ref{3551} of Lemma \ref{lem350}. We shall hereby have $\Gamma_{\mathfrak{T}}$ denote the result of the aforementioned construction with $\mathfrak{T}$ as its starting point.

As a corollary, we observe a rather strong failure of the converse of Proposition \ref{prop338}.

\begin{cor}\label{cor356}
There are local definitions $X$ and $Y$ such that $X <^M Y$ and 
\begin{equation*}
    \bigcup \ ((\mathrm{Eval}^V)" Y)  \subsetneq \bigcup \ ((\mathrm{Eval}^V)" X) \text{.}
\end{equation*}
\end{cor}

\begin{proof}
Let $\mathsf{St} := \{\mathfrak{T}_s : s \in V\}$ be as in Proposition \ref{prop326}. By (\ref{eq31}), 
\begin{equation*}
    \bigcup \ ((\mathrm{Eval}^V)" \mathsf{Fg})  \subsetneq \bigcup \ ((\mathrm{Eval}^V)" \mathsf{St}) \text{.}
\end{equation*}
Since $\mathsf{St} \subset \mathsf{\Pi^M_0}$, by Proposition \ref{prop338} and Theorem \ref{thm351}, $\mathsf{St} \leq^M \mathsf{Fg}$. We are left to show $\mathsf{Fg} \not\leq^M \mathsf{St}$. 

Choose any forcing notion $\mathbb{P}$ satisfying $V \not\in \mathrm{Eval}^V(\mathfrak{T}(\mathbb{P}))$. If $s$ is finite, then $\mathrm{Eval}^V(\mathfrak{T}_s) = \{V\} \not\subset \mathrm{Eval}^V(\mathfrak{T}(\mathbb{P}))$. Now assume $s$ is infinite, and let $f$ be a bijection from $|s|$ into $s$. Apply an argument similar to that through which (\ref{eq31}) was justified, to obtain some $r \subset \omega$ such that $L^V[r]$ is an outer model of $V$, but no outer model of $L^V[r]$ is a forcing extension of $V$. Then $V[f" r] \in \mathrm{Eval}^V(\mathfrak{T}_s)$ is an outer model of $L^V[r]$, so $\mathrm{Eval}^V(\mathfrak{T}_s) \not\subset \mathrm{Eval}^V(\mathfrak{T}(\mathbb{P}))$. We have thus proved that $\mathsf{Fg} \not\leq^M_w \mathsf{St}$, and this completes the proof.
\end{proof}

The proof of Corollary \ref{cor356} intimates that for TCIs with very simple theories, we can always construct a non-generic model. We cannot do the same for all $\Pi_2$ TCIs because of Proposition \ref{prop323}. Together, they make us wonder if a clear line can be drawn in $V$. Let 
\begin{align*}
    \mathrm{NG} := \{\mathfrak{T} \in V : \ & \mathfrak{T} \text{ is a } \Pi_2 \text{ TCI and} \\
    & \exists W \ \exists \mathcal{M} \! \in \! W \ \forall x \! \in \! W \ (W \text{ is an outer model of } V \text{ and} \\
    & \mathcal{M} \models^* \mathfrak{T} \text{ and } x \not\cong \mathcal{M} \text{ whenever } x \text{ is a } V \text{-generic model of } \mathfrak{T})\} \text{.}
\end{align*}

\begin{ques}[Lau, \cite{myself}]\label{q357}
Is $\mathrm{NG}$ definable in $V$?
\end{ques}

\subsection{A Strengthening}

In this subsection, we build on Theorem \ref{thm351} to achieve a strengthening of the statement ``$\mathsf{\Pi^M_2} \leq^M \mathsf{Fg}$''. This stronger statement appears in the form of Theorem \ref{thm368}. To start, let us recall some definitions and facts from order theory.

\begin{defi}
Let $\mathbb{P} = (P, \leq_{\mathbb{P}})$ be a forcing notion. A set $Q$ is an \emph{upward closed subset of} $\mathbb{P}$ iff $Q \subset P$ and for all $p, q \in P$, 
\begin{equation*}
    (q \in Q \text{ and } q \leq_{\mathbb{P}} p) \implies p \in Q \text{.}
\end{equation*}
\end{defi}

\begin{defi}
If $\mathbb{P} = (P, \leq_{\mathbb{P}})$ is a forcing notion and $p \in P$, we let $g_p (\mathbb{P})$ denote the set $$\{q \in P : p \not \! \! \bot_{\mathbb{P}} \ q\}.$$
\end{defi}

\begin{defi}
Let $\mathbb{P} = (P, \leq_{\mathbb{P}})$ be a forcing notion. A member $p$ of $P$ is an \emph{atom} of $\mathbb{P}$ iff $$\forall q_1 \ \forall q_2 \ ((q_1 \leq_{\mathbb{P}} p \text{ and } q_2 \leq_{\mathbb{P}} p) \implies q_1 \not \! \! \bot_{\mathbb{P}} \ q_2).$$
\end{defi}

\begin{prop}\label{gpgeneric}
If $\mathbb{P} = (P, \leq_{\mathbb{P}})$ is a forcing notion and $p$ is an atom of $\mathbb{P}$, then $g_p (\mathbb{P})$ is a $\mathbb{P}$-generic filter over $V$.
\end{prop}

\begin{proof}
If $D$ is dense in $\mathbb{P}$, then there is $q \in D$ with $q \leq_{\mathbb{P}} p$. Obviously, $q \in g_p (\mathbb{P})$. Therefore $g_p (\mathbb{P})$ is a $\mathbb{P}$-generic subset over $V$. To see that $g_p (\mathbb{P})$ is a filter, let $q_1$ and $q_2$ be members of $g_p (\mathbb{P})$. By the definition of $g_p (\mathbb{P})$, there are $r_1$ and $r_2$ such that 
\begin{itemize}
    \item $r_1 \leq_{\mathbb{P}} q_1$, 
    \item $r_1 \leq_{\mathbb{P}} p$,
    \item $r_2 \leq_{\mathbb{P}} q_2$,
    \item $r_2 \leq_{\mathbb{P}} p$.
\end{itemize}
As $p$ is an atom of $\mathbb{P}$, it must be the case that $r_1 \not \! \! \bot_{\mathbb{P}} \ r_2$, which means $q_1 \not \! \! \bot_{\mathbb{P}} \ q_2$.
\end{proof}

\begin{defi}
A forcing notion $\mathbb{P} = (P, \leq_{\mathbb{P}})$ is \emph{atomless} iff no member of $P$ is an atom of $\mathbb{P}$.
\end{defi}

A non-empty atomless forcing notion gives rise to many forcing extensions.

\begin{prop}\label{prop359}
Let $V$ be a CTM such that
\begin{equation*}
    V \models \text{``} \mathbb{P} = (P. \leq_{\mathbb{P}}) \text{ is a non-empty atomless forcing notion''.} 
\end{equation*}
Then $|\mathrm{Eval}^V(\mathfrak{T}(\mathbb{P}))| = 2^{\aleph_0}$.
\end{prop}

\begin{proof}
As all members of $\mathrm{Eval}^V(\mathfrak{T}(\mathbb{P}))$ are countable, each one of them contains only countably many $\mathbb{P}$-generic filters over $V$. By Proposition \ref{prop323}, we just need to show there are $2^{\aleph_0}$ many $\mathbb{P}$-generic filters over $V$. 

The idea is to construct a Cantor scheme differentiating the generic filters in question. Outside $V$, there are countably many dense subsets of $\mathbb{P}$, so let $\{D_n : n < \omega\}$ enumerate them. Define members of the set $\{p_s : s \in 2^{< \omega}\}$ such that
\begin{enumerate}[label=(\arabic*)]
    \item\label{3591} $p_{\emptyset} \in P$,
    \item $p_s \in D_n$ if $|s| = n + 1$,
    \item $p_{s_0} \leq_{\mathbb{P}} p_{s_1}$ if $s_1 \subset s_0$, and
    \item\label{3594} $p_{s_0} \ \bot_{\mathbb{P}} \ p_{s_1}$ if $s_1 \not\subset s_0$ and $s_0 \not\subset s_1$.
\end{enumerate}
This can be done by induction on the length of $s$. Choose any condition of $\mathbb{P}$ to be $p_{\emptyset}$. Assume next that $p_s$ has been defined. Since $p_s$ is not an atom of $\mathbb{P}$, we can find $q_0$ and $q_1$ extending $p_s$ in $\mathbb{P}$ such that $q_0 \ \bot_{\mathbb{P}} \ q_1$. The density of $D_{|s|}$ guarantees there are $q'_0, q'_1 \in D_{|s|}$ extending $q_0$ and $q_1$ in $\mathbb{P}$, respectively. Set
\begin{gather*}
    p_{s^\frown \langle 0 \rangle} := q'_0 \\
    p_{s^\frown \langle 1 \rangle} := q'_1 \text{.}
\end{gather*}
It is not hard to verify \ref{3591} to \ref{3594} hold for the $p_s$s defined as such.

Given $r \in 2^{\omega}$, use $g_r$ to denote the set
\begin{equation*}
    \{q \in P : \exists n < \omega \ (p_{r \restriction n} \leq_{\mathbb{P}} q)\} \text{.}
\end{equation*}
Now $g_r$ is a $\mathbb{P}$-generic filter over $V$ whenever $r \in 2^{\omega}$. If $r_0, r_1 \in 2^{\omega}$ and $r_0 \neq r_1$, then $r_0 \restriction n \neq r_1 \restriction n$ for some $n < \omega$. By \ref{3594} we have $p_{r_0 \restriction n} \ \bot_{\mathbb{P}} \ p_{r_1 \restriction n}$, so $g_{r_0} \neq g_{r_1}$. We are done because obviously, $|2^{\omega}| = 2^{\aleph_0}$.
\end{proof}

Models of a TCI $\mathfrak{T}$ across all outer models of $V$ can be very complicated. However, there are certain models of which atomic diagrams can be easily read off $\mathbb{P}(\mathfrak{T})$.

\begin{defi}[Lau, \cite{myself}]
Given a TCI $\mathfrak{T}$ and any $\mathcal{M}$, we say $\mathcal{M}$ is a \emph{finitely determined model of} $\mathfrak{T}$ iff $\mathcal{M} \models^* \mathfrak{T}$ and for some quantifier-free sentence $\varphi$ with parameters in $\mathcal{M}$, 
\begin{align*}
    \forall W \ \forall \mathcal{M}' \ (&(W \text{ is an outer model of } V \text{ and } \mathcal{M}' \in W \text{ and } \mathcal{M}' \models^* \mathfrak{T} \text{ and } \mathcal{M}' \models \varphi) \\ 
    & \implies \mathcal{M}' = \mathcal{M}).
\end{align*}
In this case, $\mathcal{M}$ is \emph{finitely determined by} $\varphi$.
\end{defi}

Naturally, all finite models of any TCI are finitely determined. As it turns out, if a TCI is consistent, then all its finitely determined models are already in $V$

\begin{lem}[Lau, \cite{myself}]\label{findetinV}
Let $\mathfrak{T}$ be a TCI and $\mathcal{M}$ be a finitely determined model of $\mathfrak{T}$ in some outer model of $V$. Then for some atom $p$ of $\mathbb{P}(\mathfrak{T})$, $\Sigma(\mathfrak{T}, \mathcal{M}) = g_p (\mathbb{P}(\mathfrak{T}))$. In particular, $\mathcal{M} \in V$.
\end{lem}

\begin{proof}
Let $\mathcal{M}$ be finitely determined by $\varphi$. Without loss of generality, we can assume $\varphi$ is the conjunction of a set of literals $\{l_i : i < n\}$ for some $n < \omega$. This means $$p := \{\ulcorner E(l_i) \urcorner : i < n\}$$ is an atom of $\mathbb{P}(\mathfrak{T})$. Proposition \ref{gpgeneric} tells us that $g_p (\mathbb{P}(\mathfrak{T}))$ is $\mathbb{P}(\mathfrak{T})$-generic over $V$, so necessarily $\Sigma(\mathfrak{T}, \mathcal{M}) = g_p (\mathbb{P}(\mathfrak{T}))$ by Theorem \ref{thm351}. Then according to Proposition \ref{prop322}, $\mathcal{M} \in V$ because $g_p (\mathbb{P}(\mathfrak{T})) \in V$.
\end{proof}

It is possible to have an analogue of Lemma \ref{findetinV} for models that are ``close to being finitely determined''.

\begin{defi}[Lau, \cite{myself}]\label{CB1}
Let $\mathfrak{T}$ be a TCI. Inductively define $\Gamma_{\mathfrak{T}}^{(\alpha)}$, $P(\mathfrak{T})^{(\alpha)}$ and $\mathbb{P}(\mathfrak{T})^{(\alpha)}$ for all ordinals $\alpha \leq |[\mathcal{L}_{\mathfrak{T}}]^{< \omega}|^+$ as follows:
\begin{align*}
    \Gamma_{\mathfrak{T}}^{(0)} := \ & \Gamma_{\mathfrak{T}} \text{,} \\
    P(\mathfrak{T})^{(0)} := \ & P(\mathfrak{T}) \text{,} \\
    \Gamma_{\mathfrak{T}}^{(\alpha)} := \ & \Gamma_{\mathfrak{T}}^{(\alpha - 1)} \cup \{\ulcorner \bigvee_{x \in p} (\neg E(x)) \urcorner : p \text{ is an atom of } \mathbb{P}(\mathfrak{T})^{(\alpha - 1)}\} \\
    & \text{if } \alpha \text{ is a successor ordinal,} \\
    \Gamma_{\mathfrak{T}}^{(\alpha)} := \ & \bigcup_{\beta < \alpha} \Gamma_{\mathfrak{T}}^{(\beta)} \\
    & \text{if } \alpha \text{ is a limit ordinal,} \\
    P(\mathfrak{T})^{(\alpha)} := \ & \{p \in [\mathcal{L}_{\mathfrak{T}}]^{< \omega} : \ \Vdash_{Col(\omega, |trcl(\mathfrak{A}_{\mathfrak{T}})|)} \exists \Sigma \ (``\Sigma \ \Gamma_{\mathfrak{T}}^{(\alpha)} (\mathcal{L}_{\mathfrak{T}}, \mathfrak{A}_{\mathfrak{T}})\text{-certifies } p")\} \text{,} \\
    \mathbb{P}(\mathfrak{T})^{(\alpha)} := \ & (P(\mathfrak{T})^{(\alpha)}, \leq_{\mathbb{P}(\mathfrak{T})}) \text{.}
\end{align*}
\end{defi}

By a simple cardinality argument, there must exist some $\alpha < |[\mathcal{L}_{\mathfrak{T}}]^{< \omega}|^+$ for which $\Gamma_{\mathfrak{T}}^{(\alpha)} = \Gamma_{\mathfrak{T}}^{(\alpha + 1)}$, whence $\mathbb{P}(\mathfrak{T})^{(\alpha)} = \mathbb{P}(\mathfrak{T})^{(\alpha + 1)}$.

\begin{defi}[Lau, \cite{myself}]\label{CB2}
Let $\Gamma_{\mathfrak{T}}^{\top}$ denote the unique $\Gamma$ such that $\Gamma = \Gamma_{\mathfrak{T}}^{(\alpha)} = \Gamma_{\mathfrak{T}}^{(\alpha + 1)}$ for some $\alpha < |[\mathcal{L}_{\mathfrak{T}}]^{< \omega}|^+$. Similarly, $\mathbb{P}(\mathfrak{T})^{\top} = (P(\mathfrak{T})^{\top}, \leq_{\mathbb{P}(\mathfrak{T})^{\top}})$ shall denote the unique $\mathbb{P}$ such that $\mathbb{P} = \mathbb{P}(\mathfrak{T})^{(\alpha)} = \mathbb{P}(\mathfrak{T})^{(\alpha + 1)}$ for some $\alpha < |[\mathcal{L}_{\mathfrak{T}}]^{< \omega}|^+$.
\end{defi}

It is not hard to see that $P(\mathfrak{T})^{\top}$ is an atomless upward closed subset of $\mathbb{P}(\mathfrak{T})$ and $\Gamma_{\mathfrak{T}} \subset \Gamma_{\mathfrak{T}}^{\top}$.

\begin{rem}\label{CBrem}
In constructing the $\mathbb{P}(\mathfrak{T})^{(\alpha)}$'s, we are inductively removing atoms of $\mathbb{P}(\mathfrak{T})$. These atoms correspond to isolated points in a Stone-type space generated by models of a TCI. By looking at Definition \ref{CB1} in this way, we can draw obvious parallels between $\mathbb{P}(\mathfrak{T})^{(\alpha)}$ and the $\alpha$-th-order Cantor-Bendixson derivative of a set. Such parallels culminate in $\mathbb{P}(\mathfrak{T})^{\top}$ being analogous to the ``perfect core'' of $\mathbb{P}(\mathfrak{T})$.
\end{rem}

\begin{defi}[Lau, \cite{myself}]
Given a TCI $\mathfrak{T}$ and any $\mathcal{M}$, we say $\mathcal{M}$ is an \emph{almost finitely determined model of} $\mathfrak{T}$ iff $\mathcal{M} \models^* \mathfrak{T}$ and for some $\alpha < |[\mathcal{L}_{\mathfrak{T}}]^{< \omega}|^+$ and an atom $p$ of $\mathbb{P}(\mathfrak{T})^{(\alpha)}$, $$p \subset \Sigma(\mathfrak{T}, \mathcal{M}).$$
\end{defi}

We have as our next lemma, the promised analogue of Lemma \ref{findetinV}.

\begin{lem}[Lau, \cite{myself}]\label{lem365}
Let $\mathfrak{T}$ be a TCI and $\mathcal{M}$ be an almost finitely determined model of $\mathfrak{T}$ in some outer model of $V$. Then for some $\alpha < |[\mathcal{L}_{\mathfrak{T}}]^{< \omega}|^+$ and some atom $p$ of $\mathbb{P}(\mathfrak{T})^{(\alpha)}$, $\Sigma(\mathfrak{T}, \mathcal{M}) = g_p (\mathbb{P}(\mathfrak{T})^{(\alpha)})$. In particular, $\mathcal{M} \in V$.
\end{lem}

\begin{proof}
Choose any model $\mathcal{M}$ of $\mathfrak{T}$ in an outer model of $V$. It suffices to prove by induction on $\alpha \leq |[\mathcal{L}_{\mathfrak{T}}]^{< \omega}|^+$ that
\begin{align*}
    \forall q \ \exists \beta \leq \alpha \ \exists p \ ( & (q \text{ is an atom of } \mathbb{P}(\mathfrak{T})^{(\alpha)} \text{ and } q \subset \Sigma(\mathfrak{T}, \mathcal{M})) \\
    & \implies (p \text{ is an atom of } \mathbb{P}(\mathfrak{T})^{(\beta)} \text{ and } \Sigma(\mathfrak{T}, \mathcal{M}) = g_p (\mathbb{P}(\mathfrak{T})^{(\beta)}))) \text{.}
\end{align*}
The base case where $\alpha = 0$ is just Lemma \ref{findetinV}. For the inductive case, assume $0 < \alpha \leq |[\mathcal{L}_{\mathfrak{T}}]^{< \omega}|^+$. and let $q$ be an atom of $\mathbb{P}(\mathfrak{T})^{(\alpha)}$ with $q \subset \Sigma(\mathfrak{T}, \mathcal{M})$. Then by Lemma \ref{gpgeneric} and the definition of $\mathbb{P}(\mathfrak{T})^{(\alpha)}$, either $\Sigma(\mathfrak{T}, \mathcal{M}) = g_q (\mathbb{P}(\mathfrak{T})^{(\alpha)})$ or there is $\beta' < \alpha$ and an atom $q'$ of $\mathbb{P}(\mathfrak{T})^{(\beta')}$ such that $q' \subset \Sigma(\mathfrak{T}, \mathcal{M})$. In the latter case, the inductive hypothesis gives us $\beta \leq \beta'$ and an atom $p$ of $\mathbb{P}(\mathfrak{T})^{(\beta)}$ for which $\Sigma(\mathfrak{T}, \mathcal{M}) = g_p (\mathbb{P}(\mathfrak{T})^{(\beta)})$. Either way we are done.
\end{proof}

The way $\mathbb{P}(\mathfrak{T})$ and $\mathbb{P}(\mathfrak{T})^{\top}$ are defined from a TCI $\mathfrak{T}$ allows us to establish a nice dichotomy on the $(\mathbb{P}(\mathfrak{T}), V)$-generic models of $\mathfrak{T}$ when $\mathfrak{T}$ is $\Pi_2$.

\begin{lem}[Lau, \cite{myself}]\label{lem371}
Let $\mathfrak{T}$ be a $\Pi_2$ TCI and $\mathcal{M}$ be a $(\mathbb{P}(\mathfrak{T}), V)$-generic model of $\mathfrak{T}$. Then one of the following must hold:
\begin{enumerate}[label=(\arabic*)]
    \item $\mathcal{M}$ is almost finitely determined.
    \item $\mathcal{M}$ is a $(\mathbb{P}(\mathfrak{T})^{\top}, V)$-generic model of $\mathfrak{T}$.
\end{enumerate}
\end{lem}

\begin{proof}
Let $g$ be a $\mathbb{P}(\mathfrak{T})$-generic filter over $V$ and assume $\mathcal{A} \cap g = \emptyset$, where $$\mathcal{A} := \{p : \exists \alpha \ (\alpha < |[\mathcal{L}_{\mathfrak{T}}]^{< \omega}|^+ \text{ and } p \text{ is an atom of } \mathbb{P}(\mathfrak{T})^{(\alpha)})\}.$$ This latter assumption is equivalent to saying that the unique model $\mathcal{M}$ of $\mathfrak{T}$ for which $\bigcup g = \Sigma(\mathfrak{T}, \mathcal{M})$ is not almost finitely determined. By Theorem \ref{thm351}, it suffices to show that $g$ is a $\mathbb{P}(\mathfrak{T})^{\top}$-generic filter over $V$. Clearly, $\bigcup g$ $\Gamma_{\mathfrak{T}}^{\top} (\mathcal{L}_{\mathfrak{T}}, \mathfrak{A}_{\mathfrak{T}})$-certifies $p$, so $g \subset \mathbb{P}(\mathfrak{T})^{\top}$. That $\mathbb{P}(\mathfrak{T})^{\top}$ is a suborder of $\mathbb{P}(\mathfrak{T})$ means $g$ is a filter on $\mathbb{P}(\mathfrak{T})^{\top}$.

To see $g$ is $\mathbb{P}(\mathfrak{T})^{\top}$-generic over $V$, let $E$ be predense in $\mathbb{P}(\mathfrak{T})^{\top}$. Note that if $p \in \mathbb{P}(\mathfrak{T})$ is incompatible in $\mathbb{P}(\mathfrak{T})$ with every member of $\mathcal{A}$, then $p \in \mathbb{P}(\mathfrak{T})^{\top}$. As such, $E \cup \mathcal{A}$ must be predense in $\mathbb{P}(\mathfrak{T})$. But this implies $E \cap g \neq \emptyset$ because $g$ is $\mathbb{P}(\mathfrak{T})$-generic and $\mathcal{A} \cap g = \emptyset$.
\end{proof}

The following is a stronger version of Theorem \ref{thm351}.

\begin{thm}[Lau, \cite{myself}]\label{thm367}
Let $\mathfrak{T}$ be a $\Pi_2$ TCI. Then one of the following must hold.
\begin{enumerate}[label=(\arabic*)]
    \item All models of $\mathfrak{T}$ are almost finitely determined.
    \item $\mathbb{P}(\mathfrak{T})^{\top}$ is non-empty and every $\mathbb{P}(\mathfrak{T})^{\top}$-generic filter over $\mathfrak{A}_{\mathfrak{T}}$ witnesses $\mathcal{M}$ is a $(\mathbb{P}(\mathfrak{T})^{\top}, \mathfrak{A}_{\mathfrak{T}})$-generic model of $\mathfrak{T}$ for some $\mathcal{M}$.
\end{enumerate}
\end{thm}

\begin{proof}
Assume not all models of $\mathfrak{T}$ are almost finitely determined, and let $\mathcal{M}$ be a model of $\mathfrak{T}$ not almost finitely determined in some outer model of $V$. Then $\Sigma(\mathfrak{T}, \mathcal{M})$ $\Gamma_{\mathfrak{T}}^{\top} (\mathcal{L}_{\mathfrak{T}}, \mathfrak{A}_{\mathfrak{T}})$-certifies $\emptyset$, so $\mathbb{P}(\mathfrak{T})^{\top}$ is non-empty.

Let $g$ be a $\mathbb{P}(\mathfrak{T})^{\top}$-generic filter over $V$. Check that the hypotheses of Lemma \ref{lem349} are satisfied when we have 
\begin{itemize}
    \item $V[g]$ in place of $W$,
    \item $|\mathfrak{A}_{\mathfrak{T}}|$ in place of $\lambda$,
    \item $\mathfrak{A}_{\mathfrak{T}}$ in place of $\mathfrak{A}$,
    \item $\mathcal{L}_{\mathfrak{T}}$ in place of $\mathcal{L}$,
    \item $\Gamma_{\mathfrak{T}}^{\top}$ in place of $\Gamma$,
    \item $P(\mathfrak{T})^{\top}$ in place of $P$, and
    \item $\mathbb{P}(\mathfrak{T})^{\top}$ in place of $\mathbb{P}$,
\end{itemize}
A direct application of said lemma, coupled with Lemma \ref{lem350} and the knowledge that $\Gamma_{\mathfrak{T}} \subset \Gamma_{\mathfrak{T}}^{\top}$, completes the proof. 
\end{proof}

The strengthening we were aiming for can now be proven.

\begin{thm}\label{thm368}
Fix $\mathfrak{T}^* \in \mathsf{Fg}$. Then there is $F_{\mathfrak{T}^*}$ witnessing $\mathsf{\Pi^M_2} \leq^M \mathsf{Fg}$ such that
\begin{enumerate}[label=(\arabic*)]
    \item $F_{\mathfrak{T}^*}(\mathfrak{T}) = \mathfrak{T}^*$ if $\mathfrak{T}$ is inconsistent,
    \item $F_{\mathfrak{T}^*}(\mathfrak{T}) = \mathfrak{T}((\emptyset, \emptyset))$ if $\mathfrak{T}$ is consistent and all models of $\mathfrak{T}$ are almost finitely determined, and
    \item $F_{\mathfrak{T}^*}(\mathfrak{T}) = \mathfrak{T}(\mathbb{P})$ for some non-empty atomless forcing notion $\mathbb{P}$ if $\mathfrak{T}$ is consistent and not all models of $\mathfrak{T}$ are almost finitely determined.
\end{enumerate}
\end{thm}

\begin{proof}
Define $F_{\mathfrak{T}^*}$ point-wise as follows:
\begin{align*}
    F_{\mathfrak{T}^*}(\mathfrak{T}) :=  
    \begin{cases}
        \mathfrak{T}^* & \text{if } \mathfrak{T} \text{ is inconsistent} \\
        \mathfrak{T}(\mathbb{P}(\mathfrak{T})^{\top}) & \text{otherwise,}
    \end{cases}
\end{align*}
noting Lemma \ref{lem365}, Theorem \ref{thm367} and the fact that $\mathbb{P}(\mathfrak{T})^{\top} = (\emptyset, \emptyset)$ if all models of $\mathfrak{T}$ are almost finitely determined.
\end{proof}

Notice that any $F_{\mathfrak{T}^*}$ satisfying Theorem \ref{thm368} must also satisfy
\begin{equation*}
    |\mathrm{Eval}^V(\mathfrak{T})| = |\mathrm{Eval}^V(F_{\mathfrak{T}^*}(\mathfrak{T}))|
\end{equation*}
for all $\mathfrak{T} \in dom(F_{\mathfrak{T}^*})$. As a corollary, we get a trichotomy for the number of small extensions a $\Pi_2$ TCI can pick out.

\begin{cor}
Let $V$ be a CTM and $\mathfrak{T} \in V$ be a $\Pi_2$ TCI. Then
\begin{enumerate}[label=(\arabic*)]
    \item\label{3701} $\mathrm{Eval}^V(\mathfrak{T}) = \emptyset$ if $\mathfrak{T}$ is inconsistent,
    \item\label{3702} $\mathrm{Eval}^V(\mathfrak{T}) = \{V\}$ if $\mathfrak{T}$ is consistent and all models of $\mathfrak{T}$ are almost finitely determined, and
    \item\label{3703} $|\mathrm{Eval}^V(\mathfrak{T})| = 2^{\aleph_0}$ if $\mathfrak{T}$ is consistent and not all models of $\mathfrak{T}$ are almost finitely determined.
\end{enumerate}
\end{cor}

\begin{proof}
\ref{3701} follows from the definition of $\mathrm{Eval}^V$ and what it means for a TCI to be (in)consistent. \ref{3702} follows from Lemma \ref{lem365} and \ref{3703} from Proposition \ref{prop359} and Theorem \ref{thm368}.
\end{proof}

\section{References}
\printbibliography[heading=none]

\end{document}